\documentclass[10pt]{article}
\usepackage[english]{babel}
\usepackage{amsmath}
\usepackage{amsthm}
\usepackage{amssymb}

\newtheorem{theorem}{Theorem}[section]
\newtheorem{proposition}[theorem]{Proposition}
\newtheorem{definition}[theorem]{Definition}
\newtheorem{corollary}[theorem]{Corollary}
\newtheorem{lemma}[theorem]{Lemma}
\newtheorem{remark}[theorem]{Remark}

\newcommand{\mL}{{\cal L}}

\newcommand{\R}{\mathbb{R}}

\newcommand{\beq}{\begin{equation}}
\newcommand{\eeq}{\end{equation}}

\newcommand{\C}{\mathbb{C}}
\newcommand{\N}{\mathbb{N}}

\newcommand{\mA}{{\mathcal A}}
\newcommand{\mbD}{{\mathbb{D}}}

\begin{document}

\title{\bf On maximal regularity for the Cauchy-Dirichlet mixed parabolic problem with fractional time derivative}

\author{Davide Guidetti\thanks{The author is member of GNAMPA of Istituto Nazionale di Alta Matematica}\\
\\
Dipartimento di Matematica,\\
Universit\`a di Bologna\\
Piazza di Porta S. Donato 5,\\
40126 Bologna, Italy. \\
E-mail: davide.guidetti@unibo.it
}

\date{ }

\maketitle

\begin{abstract}
We prove two maximal regularity results in spaces of continuous and H\"older continuous functions, for a mixed linear Cauchy-Dirichlet problem  with a fractional time derivative $\mathbb{D}_t^\alpha$. This derivative is intended in the sense of Caputo and $\alpha$ is taken in $(0, 2)$. In case $\alpha = 1$, we obtain maximal regularity results for mixed parabolic problems already known in mathematica literature. 
\end{abstract}


\section{Introduction} \label{se1}

The aim of this paper is the study the following mixed Cauchy-Dirichlet problem: 

\begin{equation} \label{eq1.1}
\left\{
\begin{array}{l}
\mathbb{D}^\alpha_{C(\overline \Omega)} u(t,x) =  A(x,D_x) u(t,x) + f(t,x), \quad t \in [0, T], x \in \Omega, \\ \\
u(t,x') = g(t,x') , \quad (t,x') \in [0, T] \times \partial \Omega, \\ \\
D_t^k u(0,x) = u_k(x), \quad x \in \overline \Omega, k \in \N_0, k < \alpha, 
\end{array}
\right. 
\end{equation}
with $\mathbb{D}^\alpha_{C(\overline \Omega)}u$ fractional time derivative in the sense of Caputo of order $\alpha$ in $(0, 2)$, $A(x,D_x)$ elliptic in the bounded domain $\Omega$ and Dirichlet (not necessarily homogeneous) conditions on the boundary $\partial \Omega$ of $\Omega$ (precise assumptions will be stated in the following, see (A1)-(A4)).  

Mixed boundary value problems with fractional time derivatives have attracted the attention of researchers in these latest time.  An application of a nonlinear version of (\ref{eq1.1}) to a problem in viscoelasticity is mentioned in \cite{ClLoSi1} (see also the references in this paper).  Explicit solutions in cases with simple geometries and various boundary conditions  where found in many situations (see, for example, \cite{Po1}, with its bibliography). A general discussion of several mathematical models of heat diffusion (even with fractional derivatives) is contained in \cite{FaGiMo1}. 

We prove here two maximal regularity results which are already known in the case $\alpha = 1$.

The first of these results (Theorem \ref{th1.1})  prescribes necessary and sufficient conditions on the data $f$, $g$, $u_k$ ($k \in \N_0$, $k < \alpha$), in order that, given $\theta$ in $(0, 2) \setminus \{1\}$ with $\alpha \theta < 2$,  there exists a unique solution $u$ which is continuous (as a function of $t$) with values in $C^2(\overline \Omega)$, bounded with values in $C^{2+\theta}(\overline \Omega)$ and such that $\mathbb{D}^\alpha_{C(\overline \Omega)}u$ and $A(\cdot,D_x)u$ are continuous with values in $C(\overline \Omega)$ and bounded with values in $C^{\theta}(\overline \Omega)$. The case $\alpha = 1$ was proved in \cite{Gu4} and generalized to general mixed parabolic problems in 
\cite{Gu1}. Related questions  were discussed in  \cite{SivW1}. 

The second of these results (Theorem \ref{th1.2}) prescribes necessary and sufficient conditions on the data $f$, $g$, $u_k$ ($k \in \N_0$, $k < \alpha$), in order that, given $\theta$ in $(0, 2) \setminus \{1\}$ with $\alpha \theta < 2$,  there exists a unique solution $u$ which is continuous (as a function of $t$) with values in $C^2(\overline \Omega)$, bounded with values in $C^{2+\theta}(\overline \Omega)$ and such that $\mathbb{D}^\alpha_{C(\overline \Omega)}u$ and $A(\cdot,D_x)u$ belong to the class $C^{\frac{\alpha \theta}{2}, \theta}([0, T] \times \overline \Omega)$. The case $\alpha = 1$ is classical and is completely illustrated in \cite{LaSoUr1}  and \cite{Lu1}.  See also \cite{LuSivW1} for a semigroup approach.  We are not aware of generalizations to the case $\alpha \neq 1$. 

Our study might be the starting point to consider nonlinear problems by linearization procedures.

We quote other papers connected with the content of this one.

In \cite{EiKo1} the Cauchy problem in $\R^n$ is studied in case $\alpha \in (0, 1]$.  A fundamental solution is constructed. The simplest case, namely the case with $n = 1$ and the elliptic operator with a constant coefficients, is studied in \cite{Ma1}. 

In \cite{KeLiWa1}  the authors consider the abstract Cauchy problem 
\begin{equation}\label{eq1.2A}
\left\{ \begin{array}{ll}
\mathbb{D}^\alpha_{X} u (t) = A u(t) + f(t), & t > 0, \\ \\
u(0) = u_0,
\end{array}
\right. 
\end{equation}
with $\alpha$ in $(0, 1]$ (we shall precise in Definition \ref{de2.2} the meaning of the expression $\mathbb{D}^\alpha_{X} u$). They consider the case that $A$ is the infinitesimal generator of a $\beta-$times integrated semigroup in the Banach space $X$. Their results are also applied to our problem (see their Example 8.3), with $X = C(\overline \Omega) \times C(\partial \Omega)$, but they do not seem to be of maximal regularity. 

Maximal regularity results are discussed in \cite{ClGrLo1} and \cite{ClGrLo2} for the general abstract system 
 
 \begin{equation} \label{eq1.3}
\left\{
\begin{array}{l}
\mathbb{D}^\alpha_{X} u(t) =  Au(t) + f(t), \quad t \in [0, T], \\ \\
D_t^k u(0) = u_k,  k \in \N_0, k < \alpha, 
\end{array}
\right. 
\end{equation}
in case $\alpha \in (0, 2)$, $-A$ is a sectorial operator of type less than $(1 - \frac{\alpha}{2}) \pi$ (see Definition \ref{de2.1}). Two topics are discussed: 

(I) necessary and sufficient conditions on the data, in order that $\mathbb{D}^\alpha_{X} u$ and $Au$ are bounded with values in the real interpolation space $(X, D(A))_{\theta,\infty}$ ($0 < \theta < 1$); 

(II) necessary and sufficient conditions on the data, in order that $\mathbb{D}^\alpha_{X} u$ and $Au$ are both in the space of H\"older continuous functions $C^\beta ([0, T]; X)$. 

In \cite{ClGrLo1}  the case $\alpha \in (0, 1]$ is considered. In this case the results found are essentially complete. The case $\alpha \in (1, 2)$ is considered in \cite{ClGrLo2}. Here only sufficient conditions are prescribed. 
In order to prove Theorem \ref{th1.2}, we shall consider the case (II), with $\beta < \alpha$. It turns out that the sufficient conditions prescribed in \cite{ClGrLo2} are also necessary. This is proved in the preprint \cite{Gu3}. In the following (see Theorems \ref{th2.1A}-\ref{th2.1B}) we shall come back to these results, as we shall need them. Here we mention only the fact that, given $\theta$ in (say) $(0, 1)$, the operator $-A$ such that
$$
\left\{\begin{array}{l}
D(A) = \{u \in C^{2+\theta}(\overline \Omega) : u_{|\partial \Omega} = 0\} \\ \\
Au = A(\cdot, D_x) u. 
\end{array}
\right. 
$$
is not sectorial in the Banach space $C^\theta(\overline \Omega)$ (see \cite{Lu1}, Example 3.1.33): the best available estimate is (\ref{eq2.8}). So, even in the case of homogeneous boundary conditions, the results of \cite{ClGrLo1} and  \cite{ClGrLo2} are not sufficient for our purposes.

Other results of maximal regularity for (\ref{eq1.3}) are discussed in \cite{ClLoSi1} and in \cite{Ba1} (see also \cite{Ba2}). Finally, maximal regularity for equations involving versions of the Caputo derivative in $\R$, in spaces of order continuous functions on the line are given in  \cite{Po2} and \cite{KeLi1}.

Now we introduce some notations which we are going to use in the paper.

 If $\alpha \in \R$,  $[\alpha]$ will indicate the maximum integer less or equal than $\alpha$. $\R^+$ will indicate the set of (strictly) positive real numbers. If $\lambda \in \C \setminus \{0\}$.
 
  If $X$ is a complex Banach space with norm $\|\cdot\|$,  $B([0, T]; X)$ will indicate the class of functions with values in $X$ with domain $[0, T]$; if $B$ is a (generally unbounded) linear operator from 
$D(B) \subseteq X$ to $X$, $\rho (B)$ will indicate the resolvent set of $B$. If $A$ is a closed operator in $X$, $A : D(A) (\subseteq X) \to X$, $D(A)$, equipped with the norm
$$
\|x\|_{D(A)} : = \|x\| + \|Ax\|
$$
is a Banach space. 

Given a function $f$ with domain $\overline \Omega$, with $\Omega \subseteq \R^n$, $\gamma f$ will indicate the trace of $f$ on the boundary $\partial \Omega$ of $\Omega$. 

If $X_0, X_1$ are Banach spaces such that $X_1 \hookrightarrow X_0$, $\xi \in (0, 1)$ and $p \in [1, \infty]$, we shall indicate with $(X_0, X_1)_{\xi,p}$ the corresponding real interpolation space. We shall freely use the basic facts concerning real interpolation theory (see, for example, \cite{Lu1}, \cite{Tr2}). If $X_0 \hookrightarrow X \hookrightarrow X_0$, we shall write $X \in J_\xi(X_0,X_1)$,  $X \in K_\xi(X_0,X_1)$  if $X \hookrightarrow (X_0, X_1)_{\xi,\infty}$. 

If $\beta \in \N_0$ and $\Omega$ is an open, bounded subset of $\R^n$, we shall indicate with $C^\beta(\overline \Omega)$ the class of complex valued  functions which are continuous in $\overline \Omega$, together with their derivatives (extensible by continuity to $\overline \Omega$) of order not exceeding $\beta$. If $\beta \in \R^+ \setminus \N$, $C^\beta(\overline \Omega)$ will indicate the class of functions in $C^{[\beta]}(\overline \Omega)$ whose derivatives of order $[\beta]$ are H\"older continuous of order $\beta - [\beta]$ in $\overline \Omega$. These definitions admit natural extensions to function with values in a Banach space $X$. In this case, we shall use the notation 
$C^\beta (\overline \Omega; X)$ (in particular $C^\beta ([a, b]; X)$, in case $\Omega = (a, b) \subseteq \R$). By local charts, if $\partial \Omega$ is sufficiently regular, we can consider the spaces $C^\beta(\partial \Omega)$. All these classes will be assumed to be equipped of natural norms. We shall use the notation
$$
C^\beta_0(\overline \Omega):= \{f \in C^\beta(\overline \Omega): \gamma f = f_{|\partial \Omega} = 0\}. 
$$
If $\alpha, \beta \in [0, \infty)$, $T \in \R^+$ and $\Omega$ is an open bounded subset of $\R^n$, we set
$$
C^{\alpha,\beta}([0, T] \times \overline \Omega):= C^\alpha([0, T]; C(\overline \Omega)) \cap B([0, T]; C^\beta(\overline \Omega)). 
$$
An analogous meaning will have $C^{\alpha,\beta}([0, T] \times \partial \Omega)$. If $X$ is a Banach space, $Lip([0, T]; X)$ will indicate the class of Lispchitz continuous functions from $[0, T]$, equipped with a natural norm. 

Let $\phi \in (0, \pi)$,  $R \in [0, \infty)$. We shall indicate with $\Gamma (\phi,R)$ a piecewise $C^1$ path, describing 
$$\{\lambda \in \C : |\lambda| \geq R, |Arg(\lambda)| = \phi\} \cup \{\lambda \in \C : |\lambda| = R, |Arg(\lambda)| \leq \phi\},$$k
oriented from $\infty e^{-i\phi}$ to $\infty e^{i\phi}$. We shall write $\lambda \in \Gamma (\phi,R)$ to indicate that $\lambda$ belongs to the range of $\Gamma (\phi,R)$. 

Finally, $C$ will indicate a positive real constant we are not interested to precise (the meaning of which may be different from time to time). In a sequence of inequalities, we shall write $C_1, C_2, \dots$. 

After these preliminaries, we list the basic assumptions we are going to work with. We assume that: 

{\it \medskip

(A1) $\Omega$ is an open, bounded subset in $\R^n$ lying on one side of its boundary $\partial \Omega$, which is a $n-1-$submanifold of $\R^n$ of class $C^{2+\theta}$, with $\theta \in (0, 2) \setminus \{1\}$. 

\medskip

(A2) $\alpha \in (0, 2)$, $A(x,D_x) = \sum_{|\rho| \leq 2} a_\rho(x) D_x^\rho$, with $a_\rho \in C^\theta(\overline \Omega)$, $a_\rho$ complex valued; $A(x,D_x)$ is assumed to be elliptic, in the sense that $\sum_{|\rho| = 2} a_\rho(x) \xi^\rho \neq 0$
$\forall \xi \in \R^n \setminus \{0\}$; we suppose, moreover, that
$$
|Arg( \sum_{|\rho| = 2} a_\rho(x) \xi^\alpha)| <  (1- \frac{\alpha}{2})\pi, \quad \forall x \in \overline \Omega, \forall \xi \in \R^n \setminus \{0\}. 
$$

\medskip

(A3)  $\mathbb{D}^\alpha_{C(\overline \Omega)} u$ is the fractional Caputo derivative of $u$ with respect to $t$ with values in $C(\overline \Omega)$ (see the following Definition \ref{de2.2}). 

\medskip

\medskip

(A4) $\alpha \theta < 2$. 

}

\medskip

We are looking for necessary and sufficient conditions in order that (\ref{eq1.1}) has a unique solution $u$ such that: 

\medskip
{\it

 (B1) $\mathbb{D}^\alpha_{C(\overline \Omega)} u$  exists and belongs to $C([0, T]; C(\overline \Omega)) \cap B([0, T]; C^\theta(\overline \Omega))$. 
 
 \medskip
 
 (B2) $u \in C([0, T]; C^2(\overline \Omega)) \cap B([0, T]; C^{2+\theta}(\overline \Omega))$. 
 
 }
 
 \medskip
 
 We want to prove the following

  \begin{theorem}\label{th1.1}
 Suppose that the assumptions (A1)-(A4) are fulfilled. Then the following conditions are necessary and sufficient, in order that (\ref{eq1.1}) has a unique solution $u$ satisfying (B1)-(B2): 
 
 (I) $f \in C([0, T]; C(\overline \Omega)) \cap B([0, T]; C^\theta(\overline \Omega))$.
 
 (II) $u_0 \in C^{2+\theta}(\overline \Omega)$ and, in case $\alpha \in (1, 2)$, $u_1 \in C^{\theta + 2(1 - \frac{1}{\alpha})}(\overline \Omega)$. 
 
 (III) $g \in C([0, T]; C^2(\partial \Omega)) \cap B([0, T]; C^{2+\theta}(\partial \Omega))$, $\mathbb{D}^\alpha_{C(\partial \Omega)} g$ exists and belongs to $C([0, T]; C(\partial \Omega)) \cap B([0, T]; C^\theta(\partial \Omega))$; 
 
 (IV) $\gamma u_0 = g(0)$ and, in case $\alpha \in (1, 2)$, $\gamma u_1 = D_t g(0)$ 
 
 (V) $\gamma f - \mathbb{D}^\alpha_{C(\partial \Omega)} g \in C^{\frac{\alpha \theta}{2}} ([0, T]; C(\partial \Omega))$. 
 
 (VI) $\gamma[A(\cdot,D_x) u_0 + f(0)] = \mbD^\alpha_{C(\partial \Omega)} g(0)$. 

 \end{theorem}
 
 We are also looking for necessary and sufficient conditions in order that (\ref{eq1.1}) has a unique solution $u$ such that: 

\medskip
{\it (D1) $u \in C([0, T]; C^2(\overline \Omega)) \cap B([0, T]; C^{2+\theta}(\overline \Omega))$; 
 
 \medskip
 
 (D2) 
 }
 
 \medskip
 
  We want to prove the following

  \begin{theorem}\label{th1.2}
 Suppose that the assumptions (A1)-(A4) are fulfilled. Then the following conditions are necessary and sufficient, in order that (\ref{eq1.1}) has a unique solution $u$ satisfying (D1)-(D2): 
 
 (I) $f \in C^{\frac{\alpha \theta}{2}, \theta}([0, T] \times \overline \Omega)$.
 
 (II) $u_0 \in C^{2+\theta}(\overline \Omega)$ and, in case $\alpha \in (1, 2)$, $u_1 \in C^{\theta + 2(1 - \frac{1}{\alpha})}(\overline \Omega)$. 
 
 (III) $g \in C([0, T]; C^2(\partial \Omega)) \cap B([0, T]; C^{2+\theta}(\partial \Omega))$, $\mathbb{D}^\alpha_{C(\partial \Omega)} g$ exists and belongs to $C^{\frac{\alpha \theta}{2}, \theta}([0, T] \times \partial \Omega)$; 
 
 (IV) $\gamma u_0 = g(0)$ and, in case $\alpha \in (1, 2)$, $\gamma u_1 = D_t g(0)$. 

(V) $\gamma[A(\cdot,D_x) u_0 + f(0)] = \mbD^\alpha_{C(\partial \Omega)} g(0)$. 

 \end{theorem}
 
 The paper is organized in the following way: Section \ref{se2} contains a series of preliminaries results that we shall use in the sequel. In particular, we have put here the definition of Caputo derivative (Definition \ref{de2.2}), with a description of its main properties. In the final part we briefly discuss we abstract system (\ref{eq1.3}), with reference to the results in \cite{ClGrLo1} and \cite{ClGrLo2}. Section \ref{se3} contains a proof of Theorem \ref{th1.1}. Finally, Section \ref{se4} is dedicated to the proof of Theorem \ref{th1.2}.

\section{Preliminaries} \label{se2}

\setcounter{equation}{0}

We begin by recalling with some properties of the class of spaces $C^\beta (\overline \Omega)$ ($0 \leq \beta \leq 2+\theta$). 

\begin{lemma}\label{le2.1}
Let $\Omega$ be an open subset in $\R^n$ as in (A1). Let $0 \leq \beta_0 < \beta_1 \leq 2+\theta$. Then:

(I) if $\xi \in (0, 1)$,  
$$
C^{(1-\xi)\beta_0 + \xi \beta_1}(\overline \Omega) \in J_\xi(C^{\beta_0}(\overline \Omega), C^{\beta_1}(\overline \Omega)) \cap K_\xi (C^{\beta_0}(\overline \Omega), C^{\beta_1}(\overline \Omega)) ; 
$$
(II) if $(1 - \xi) \beta_0 + \xi \beta_1 \not \in \N$,
$$
C^{(1-\xi)\beta_0 + \xi \beta_1}(\overline \Omega) = (C^{\beta_0}(\overline \Omega), C^{\beta_1}(\overline \Omega))_{\xi,\infty}
$$
with equivalent norms.

(III) If $\beta \in (0, 2+\theta] \setminus \N$, any bounded and closed subset of $C^\beta(\overline \Omega)$ is closed in $C(\overline \Omega)$.

(IV) There exists an element $R$ if $\mL(C(\partial \Omega), C(\overline \Omega))$ such that $\gamma Rg = g$ $\forall g \in C(\partial \Omega)$ and, for any $\xi$ in $[0, 2+\theta]$,  $R_{|C^\xi(\partial \Omega)}$ belongs to $\mL(C^\xi(\partial \Omega), C^\xi(\overline \Omega))$. 
\end{lemma}

\begin{proof}
See \cite{Gu1}, Proposition 1.1 for (I)-(III). Concerning (IV), we construct an operator with similar properties in $\R^n_+$, setting, for $g \in C(\R^{n-1})$, 
$$
R_0g(x',x_n):= g(x') \chi (x_n), \quad (x',x_n) \in \R^n_+,
$$
with $\chi \in C^\infty([0, \infty))$, $\chi(t) = 1$ if $t \in [0, \delta_1]$, $\chi(t) = 1$ if $t \geq \delta_2$, $0 < \delta_1 < \delta_2$. $R$ can be constructed employing $R_0$, local charts and a partition of unity. 
\end{proof}
 
 We introduce the definition and some properties of the Caputo derivative $\mbD^\alpha_X u$. We shall consider the case of functions $u$ defined in $[0, T]$ with values in the complex 
 Banach space $X$. The definition requires some preliminaries. We start from the following simple operator $B$: 
 \begin{equation}\label{eq1.2}
 \left\{\begin{array}{l}
 B_X: \{v \in C^1([0, T]; X) : v(0) = 0\} \to C([0, T]; X), \\ \\
 B_X v:= D_tv. 
 \end{array}
 \right. 
 \end{equation}
Then $\rho(B_X) = \C$ and $B_X$ is a positive operator in $C([0, T]; X)$  of type $\frac{\pi}{2}$ in the sense of the following definition: 
 
 \begin{definition}\label{de2.1} Let $B$ be a linear operator in  the complex Banach space $Y$. We shall say that $B$ is positive of type $\omega$, with $\omega \in (0, \pi)$, if 
 $$
 \{\lambda \in \C \setminus \{0\} : |Arg(\lambda)| >  \omega\} \cup \{0\} \subseteq \rho (B). 
 $$
 Moreover, $\forall \epsilon \in (0,  \pi - \omega)$, there exists $M(\epsilon)$ positive such that 
 
 $$\|\lambda (\lambda - A)^{-1}\|_{\mL(Y)} \leq M(\epsilon)$$
  in case $\lambda \in \C \setminus \{0\}$,  $|Arg(\lambda)| \geq  \omega + \epsilon$.
\end{definition}
We pass to define the powers of a positive operator. For the definition of positive operator see \cite{Ta1}, Definition 2.3.1, where also the condition that $D(B)$ is dense in $Y$ is requires. In order to describe and prove the properties if the fractional properties of $B_X$, we shall appeal (if possible) to corresponding  results in \cite{Ta1}, concerning fractional powers of positive operators with dense domain. 
If $B$ is a positive operator in $X$ of type $\omega$, and $\alpha \in \R^+$, we set
\begin{equation}
B^{-\alpha}: = -\frac{1}{2\pi i} \int_{\Gamma(\phi,R)} \lambda^{-\alpha} (\lambda - B)^{-1} d\lambda. 
\end{equation}
with $\phi \in (\omega, \pi)$ and $R$ positive, such that $\{\lambda \in \C : |\lambda| \leq R\} \subseteq \rho(B)$. It turns out (applying standard computations techniques of complex integrals) that,  we have, $\forall \alpha \in \R^+$, $\forall f \in C([0, T]; X)$, $\forall t \in [0, T]$: 
\begin{equation}\label{eq1.5}
 B_X^{-\alpha} f(t) = \frac{1}{\Gamma (\alpha)} \int_0^t (t-s)^{\alpha -1} f(s) ds. 
 \end{equation}
 With arguments similar to those employed in \cite{Ta1}, Chapter 2.3, one can show the following
 
 \begin{lemma}\label{le1.2}
 Let $B_X$ be the operator defined in (\ref{eq1.2}) and let $\alpha, \beta \in \R^+$. Then: 
 
 (a) $B_X^{-\alpha} B_X^{-\beta} = B_X^{- (\alpha +\beta)}$. 
 
 (b) (\ref{eq1.5}) is consistent with the usual definition of $B_X^{-\alpha}$ in case $\alpha \in \N$. 
 
 (c) $B_X^{-\alpha}$ is injective. 
 \end{lemma} 
 So we can define, for any $\alpha$ in $\R^+$, 
 \begin{equation}
 B_X^\alpha:= (B_X^{-\alpha})^{-1}. 
 \end{equation}
 Of course the domain $D(B_X^\alpha)$ of $B_X^\alpha$ is the range of $B_X^{-\alpha}$.  
 \begin{lemma}\label{le1.2A}
 Let $\alpha, \beta \in \R^+$. Then: 
 
 (a) $D(B_X^{\alpha + \beta}) = \{u \in D(B_X^\alpha) : B_X^\alpha u \in D(B_X^\beta)\}$; moreover, if $u \in D(B_X^{\alpha + \beta})$, $B_X^{\alpha + \beta}u = B_X^\beta(B_X^\alpha u)$; 
 
 (b) if $\alpha \leq \beta$, $D(B_X^\beta) \subseteq D(B_X^\alpha)$; 
 
 (c) if $\alpha \in \N$, 
 $$
 D(B_X^\alpha) = \{u \in C^\alpha([0, T]; X) : u^{(k)}(0) = 0 \quad \forall k \in \N_0, k < \alpha\}. 
 $$
 (d) If $\alpha \in  \R^+$, $\rho(B_X^\alpha) = \C$ and $\forall \lambda \in \C$, $\forall f \in C([0, T]; X)$, $\forall t \in [0, T]$
 \begin{equation}\label{eq2.6A}
[(\lambda - B_X^\alpha)^{-1}f](t) =  - \frac{1}{2\pi i} \int_{\Gamma(\phi,R)} (\lambda - \mu^\alpha)^{-1} (\mu - B)^{-1} d\mu = \frac{1}{2\pi i} 
 \int_0^t ( \int_{\Gamma(\phi,R)} \frac{e^{\mu (t-s)}}{\lambda - \mu^\alpha} d\mu) f(s) ds, 
 \end{equation}
 with $\phi \in (\frac{\pi}{2}, \pi)$, $R^\alpha > |\lambda|$. 
 
 (e) If $\alpha \in (0, 2)$, $B_X^\alpha$ is positive of type $\frac{\alpha \pi}{2}$. 
   \end{lemma}
 
 \begin{proof} Concerning (a), (b), one can follow the arguments in \cite{Ta1}, Chapter 2.3). 
 
 (c) is trivial. 
 
 Concerning (d), let $\lambda \in \C$. Let $\Gamma(\lambda) = \Gamma(\phi, R)$, with $\phi$ and $R$ as in the statement. If $f \in C([0, T]; X)$, we set
$$
T(\lambda):= - \frac{1}{2\pi i} \int_{\Gamma(\lambda)} (\lambda - \mu^\alpha)^{-1} (\mu - B)^{-1} d\mu. 
$$
 It is easily seen (observing that  $\Gamma(\lambda)$  can be chosen locally independently of $\lambda$) that $T$
   is entire with values in $\mL(C([0, T]; X))$. By well known facts of analytic continuation, in order to show that $T(\lambda) = (\lambda - B^\alpha)^{-1}$, it is sufficient to show that this holds if $\lambda$ belongs to some ball centred in $0$. We set
$R(\lambda):= (\lambda - B^\alpha)^{-1}$, with $\lambda$ sufficiently close to $0$, in such a way that it belongs to $\rho(B^\alpha)$ (as $0 \in \rho(B_X^\alpha)$). We prove that $T^{(k)}(0) = R^{(k)}(0)$ for every $k \in \N_0$. In fact, we have
$$
R^{(k)}(0) = - k! B^{-(k+1)\alpha}.
$$
On the other hand, 
$$
T^{(k)}(0) =  \frac{ k!}{2\pi i} \int_{\Gamma(0)}  \mu^{-\alpha(k+1)} (\mu - B)^{-1} d\mu = - k! B^{-(k+1)\alpha},
$$
and the conclusion follows. 

(e) We consider first the case $\alpha \in (0, 1)$. Employing formula (\ref{eq2.6A}), we can follow the argument in \cite{Ta1}, Proposition 2.3.2 and get the conclusion in this case. If $\alpha \in (1, 2)$, $\lambda \in \C \setminus \{0\}$ and $|Arg(\lambda)| > \frac{\alpha \pi}{2}$, then $|Arg(\pm \lambda^{1/2})| > \frac{\alpha \pi}{4}$ (on account of $\frac{\alpha \pi}{4} < \frac{\pi}{2}$). We deduce that 
$$
\lambda - B_X^\alpha = (\lambda^{1/2} + B_X^{\frac{\alpha}{2}}) (\lambda^{1/2} - B_X^{\frac{\alpha}{2}}),
$$
implying easily the conclusion. 
\end{proof}

\begin{remark}\label{re2.2}
{\rm  Let $\alpha \in (0, 2)$. Then, as $\rho(B^\alpha) = \C$, it is easily seen that for every $c$ in $\C$ $c + B_X^\alpha$ is positive of type $\frac{\alpha \pi}{2}$. 
}
\end{remark}

Now we examine the domain $D(B^\alpha_X)$ of $B^\alpha_X$. We shall employ the following

\begin{proposition} \label{pr2.1A}%pr2.1
Let $X$ be a complex Banach space, $B$ a linear operator in $X$ such that, for some $\phi \in (-\pi, \pi]$, there exists $R$ positive such that
$$
\{\lambda \in \C : |\lambda| \geq R, Arg(\lambda) = \phi\} \subseteq \rho (B)
$$
and, for some $C$ positive and $\lambda$ in this set, 
$$
\|(\lambda - B)^{-1}\|_{\mL(X)} \leq C|\lambda|^{-1}. 
$$
Then: 

(I) if $j, k \in \N$ and $j < k$, $D(B^j) \in K_{j/k}( X, D(B^k)) \cap J_{j/k}( X, D(B^k))$; 

(II) if $j_0$, $k_0$, $j_1$, $k_1$ are nonnegative integers such that $j_0 < k_0$, $j_1 < k_1$, $\xi_0, \xi_1 \in (0, 1)$, $(1 - \xi_0)j_0 + \xi_0 k_0 = (1 - \xi_1)j_1 + \xi_1 k_1$,   then, for any $p$ in $[1, \infty]$, 
$$
(D(B^{j_0}), D(B^{k_0}))_{\xi_0,p} = (D(B^{j_1}), D(B^{k_1}))_{\xi_1,p}. 
$$

(III) If $\xi \in (0, 1)$, 
$$(X, D(B))_{\xi,\infty} = \{f \in X : \limsup_{t \to \infty} t^\xi \|B(te^{i\phi} -B)^{-1}f\| < \infty\}. $$
\end{proposition}

\begin{proof}
For (I) see \cite{Tr2}, Chapter 1.14.3. (II) follows from (I) and the reiteration property of the real method. (III) is proved in \cite{Gr1}, Theorem 3.1. 
\end{proof}
 
 \begin{lemma} \label{le1.3}
 Let $X, X_1$ be complex Banach spaces, such that $X_1 \hookrightarrow X$ and closed, bounded subsets of $X_1$ are also closed in $X$. Let $\alpha \in \R^+$,  $u \in D(B^\alpha_X)$ and suppose that $B_X^\alpha u \in B([0, T]; X_1)$. Then: 
 
 (I) if $\alpha \not \in \N$, then $u \in C^\alpha([0, T]; X_1)$. 
 
 (II) If $\alpha \in \N$, then $u \in C^{\alpha-1}([0, T]; X_1)$ and $D_t^{\alpha -1} u \in Lip([0, T]; X_1)$. 
 
 (III) If $k \in \N_0$ and $k < \alpha$, $D_t^k u(0) = 0$. 
  \end{lemma}
  
  \begin{proof} (I) Suppose first that $\alpha \in (0, 1)$. Let $f := B_X^\alpha u$. Then $f \in C([0, T]; X) \cap B([0, T]; X_1)$ and, if $t \in [0, T]$, 
$$
u(t) = \frac{1}{\Gamma(\alpha)} \int_0^t (t-\tau)^{\alpha-1} f(\tau) d\tau. 
$$
Then, clearly, $u(0) = 0$. Moreover, if $0 \leq s < t \leq T$, 
$$
u(t) - u(s) =  \frac{1}{\Gamma(\alpha)} \int_s^t (t-\tau)^{\alpha-1} f(\tau) d\tau + \frac{1}{\Gamma(\alpha)} \int_0^s [ (t-\tau)^{\alpha-1}  - (s - \tau)^{\alpha - 1}]f(\tau) d\tau = I_1 + I_2. 
$$
We set, for $n \in \N$, $f_n : [0, T] \to X_1$, $f_n(0) = 0$,  $f_n(t) = f(\frac{kT}{n})$ if $t \in (\frac{(k-1)T}{n}, \frac{kT}{n}]$, $1 \leq k \leq n$. Then $\int_s^t (t-\tau)^{\alpha-1} f_n(\tau) d\tau \in X_1$ and
$$
\|\int_s^t (t-\tau)^{\alpha-1} f_n(\tau) d\tau\|_{X_1} \leq \frac{(t-s)^\alpha}{\alpha} \|f\|_{B([0, T]; X_1)}. 
$$
As 
$$
{\displaystyle \lim_{n \to \infty}} \|\int_s^t (t-\tau)^{\alpha-1} [f_n(\tau)- f(\tau)] d\tau\|_{X} = 0, 
$$
we deduce that $I_1 \in X_1$ and 
$$\|I_1\|_{X_1} \leq \frac{(t-s)^\alpha}{\Gamma (\alpha + 1)} \|f\|_{B([0, T]; X_1)}. $$
Analogously, one can show that $I_2 \in X_1$ and 
$$
\begin{array}{c}
\|I_2\|_{X_1} \leq  \frac{1}{\Gamma(\alpha)} \int_0^s  [(s-\tau)^{\alpha-1}  - (t - \tau)^{\alpha - 1}]  d\tau \|f\|_{B([0, T]; X_1)} \\ \\
= \frac{1}{\Gamma(\alpha +1)} [(t-s)^\alpha - (t^\alpha - s^\alpha)]  \|f\|_{B([0, T]; X_1)} \leq \frac{1}{\Gamma(\alpha +1)} (t-s)^\alpha  \|f\|_{B([0, T]; X_1)}. 
\end{array}
$$
We assume now that $\alpha > 1$. We set $f:= B^{[\alpha]} u$. Then $B^{\alpha - [\alpha]}f \in C([0, T]; X) \cap B([0, T]; X_1)$, so that $f \in C^{\alpha - [\alpha]}([0, T]; X_1)$. So the claim follows from the identity
$$
u(t) = \frac{1}{([\alpha] - 1)!} \int_0^t (t-s)^{[\alpha] - 1} f(s) ds, \quad t \in [0, T]. 
$$
Similarly, one can show (II). (III) follows from Lemma \ref{le1.2A} (b)-(c). 
  \end{proof}
  
  \begin{proposition}\label{pr2.1}
  (I) Let $\alpha \in \R^+ \setminus \N$. Then
  $$D(B_X^\alpha) \in K_{\alpha - [\alpha]}(D(B_X^{[\alpha]}), D(B_X^{[\alpha]+1})) \cap J_{\alpha - [\alpha]}(D(B_X^{[\alpha]}), D(B_X^{[\alpha]+1})). $$
  
  (II) If $0 \leq \alpha_0 < \alpha_1$, $\xi \in (0, 1)$ and $(1 - \xi)\alpha_0 + \xi \alpha_1 \not \in \N$, then 
  $$
  (D(B_X^{\alpha_0}), D(B_X^{\alpha_1})_{\xi, \infty} = \{f \in C^{(1 - \xi)\alpha_0 + \xi \alpha_1}([0, T]; X) : f^{(k)}(0) = 0 \quad \forall k \in \N_0, k < (1 - \xi)\alpha_0 + \xi \alpha_1\},
  $$
  with equivalent norms.
  \end{proposition}
    \begin{proof} (I) Using the fact that, for any $k \in \N_0$, $\alpha \in [0, \infty)$, $B_X^k$ is an isomorphism between $D(B_X^{\alpha + k})$ and $D(B_X^{\alpha})$,  it is clear that the claim in the general case follows from the particular case $0 < \alpha < 1$, . First of all, it is well known that, in this case,
  \begin{equation}\label{eq2.6B}
  (C([0, T]; X), D(B_X))_{\alpha, \infty} = \{f \in C^\alpha([0, T]; X) : f(0) = 0\},
  \end{equation}
  with equivalent norms. This is proved in \cite{DPGr1}, Appendix, if we replace $C([0, T]; X)$ with $C_0([0, T]; X):= \{f \in C([0, T]; X): f(0)\}$. If $f \in (C([0, T]; X), D(B_X))_{\alpha, \infty}$, $f$ belongs to the closure of $D(B)$ in $C([0, T]; X)$, so that necessarily $f(0) = 0$. 
  
  This implies
  (applying the definition of $(C([0, T]; X), D(B_X))_{\alpha, \infty}$ by the $K$ method (see \cite{Lu1}, Chapter 1.2.1) that 
 $$(C([0, T]; X), D(B_X))_{\alpha, \infty} = (C_0([0, T]; X), D(B))_{\alpha, \infty}.$$
 So the fact that $D(B_X^\alpha) \in K_{\alpha - [\alpha]}(C([0, T]; X), D(B_X))$ follows from Lemma \ref{le1.3} and (\ref{eq2.6B}). By Proposition 1.2.13 in \cite{Lu1}, in order to prove that $D(B_X^\alpha) \in J_{\alpha - [\alpha]}(C([0, T]; X)$$D(B_X))$ (again in case $0 < \alpha < 1$), it suffices to show that there exists $C$ positive such that, if $f \in D(B_X)$, 
 $$
 \|B^\alpha_X f\|_{C([0, T]; X)} \leq C \|f\|_{C([0, T]; X)}^{1-\alpha}  \|B_X f\|_{C([0, T]; X)}^{\alpha} 
 $$
 This can be shown following the argument in \cite{Ta1}, Proposition 2.3.3. 
 
 (II) Let $m, n \in \N_0$, with $m < \alpha_0 < \alpha_1 < n$. Then, by (I), Proposition \ref{pr2.1A} (I) and the reiteration property, if $j \in \{0,1\}$, $D(B_X^{\alpha_j}) \in J_{\frac{\alpha_j - m}{n-m}} (D(B_X^m), D(B_X^n)) \cap K_{\frac{\alpha_j - m}{n-m}} (D(B_X^m), D(B_X^n))$. So, again by the reiteration theorem, 
 $$
 \begin{array}{c}
 (D(B_X^{\alpha_0}), D(B_X^{\alpha_1}))_{\xi, \infty} = (D(B_X^m), D(B_X^n))_{\frac{(1-\xi ) \alpha_0 + \xi \alpha_1 - m}{n-m}, \infty} = (D(B_X^k), D(B_X^{k+1}))_{(1-\xi ) \alpha_0 + \xi \alpha_1 - k, \infty}
 \end{array}
 $$
 if $k = [(1-\xi ) \alpha_0 + \xi \alpha_1]$. By the interpolation property, $B_X^k$ is an isomorphism of Banach spaces between $(D(B_X^k), D(B_X^{k+1}))_{(1-\xi ) \alpha_0 + \xi \alpha_1 - k, \infty}$ and 
 $(X, D(B_X))_{(1-\xi ) \alpha_0 + \xi \alpha_1 - k, \infty}$. So
 $$
 (D(B_X^k), D(B_X^{k+1}))_{(1-\xi ) \alpha_0 + \xi \alpha_1 - k, \infty} = \{f \in D(B_X^k): f^{(k)} \in C^{(1-\xi ) \alpha_0 + \xi \alpha_1 - k}([0, T]; X), f^{(k)}(0)\},
 $$
 which implies (II). 
   \end{proof}
Now we prove that functions which are representable in a certain way belong to $D(B_X^\alpha)$: 

\begin{proposition} \label{pr7A} Let $\phi_0 \in (\frac{\pi}{2}, \pi)$, $R \in \R^+$, $\alpha \in [0, \infty)$. Let $F : \{\lambda \in \C : |\lambda| > R, |Arg(\lambda)| < \phi_0\} \to X$ be such that: 

(a) $F$ is holomorphic; 

(b) there exists $M \in \R^+$ such that $\|F(\lambda)\| \leq M |\lambda|^{-1-\alpha}$, if $\lambda \in \C$,  $|\lambda| > R$, $|Arg(\lambda)| < \phi_0$; 

(c) for some $F_0 \in X$, ${\displaystyle \lim_{|\lambda| \to \infty}} \lambda^{1+\alpha} F(\lambda) = F_0$. Let $R' > R$, $\frac{\pi}{2} < \phi_1 < \phi_0$. 

We set, for $t \in (0, T]$, 
 $$
 u(t) := \left\{\begin{array}{lll}
  \frac{1}{2\pi i} \int_{\Gamma(R',\phi_1)} e^{\lambda t} F(\lambda) d\lambda & {\rm if } & 0 < t \leq T, \\ 
  0 & {\rm if } & t = 0, \alpha > 0, \\ 
  F_0 & {\rm if } & t = 0, \alpha = 0. 
  \end{array}
 \right.
 $$
 Then $u \in D(B_X^\alpha)$, for $t \in (0, T]$, 
 $$
 B_X^\alpha u(t) = \frac{1}{2\pi i} \int_\gamma e^{\lambda t} \lambda^{\alpha} F(\lambda) d\lambda
 $$
 and 
 $$
 B_X^\alpha u(0) = F_0. 
 $$
\end{proposition}

\begin{proof}
We begin by considering the case $\alpha = 0$. It is clear that $u \in C((0, T]; X)$. We show that
$$
{\displaystyle \lim_{t \to 0}} u(t) = F_0. 
$$
By standard properties of holomorphic functions, we have, for $t \in (0, \min\{1, T\}]$, 
$$
\begin{array}{c}
u(t) = \frac{1}{2\pi i} \int_{\Gamma(R' t^{-1}, \phi_1)} e^{\lambda t} F(\lambda) d\lambda 
= \frac{1}{2\pi i} \int_{\Gamma(R', \phi_1)} \frac{e^{\lambda}}{\lambda} t^{-1} \lambda F(t^{-1} \lambda) d\lambda \\ \\
= F_0 + \frac{1}{2\pi i} \int_{\Gamma(R', \phi_1)}  \frac{e^{\lambda}}{\lambda} [ t^{-1} \lambda F(t^{-1} \lambda) - F_0] d\lambda
\end{array}
$$
and the second summand vanishes as $t \to 0$, by the dominated convergence theorem. 

Suppose now that $\alpha > 0$. We set, for $t \in (0, T]$, 
$$
f(t) = \frac{1}{2\pi i} \int_\gamma e^{\lambda t} \lambda^{\alpha} F(\lambda) d\lambda .
$$
Then, employing what we have seen in case $\alpha = 0$, we deduce $f \in C([0, T]; X)$ and $f(0) = F_0$. We check that
$$
B_X^{-\alpha}f= u. 
$$
 In fact, if we put
 $$
 v(t) = B_X^{-\alpha}f(t) = \frac{1}{\Gamma (\alpha)} \int_0^t (t-s)^{\alpha - 1} f(s) ds, 
 $$
 and we consider the extensions of $u$, $f$, $v$ to $[0, \infty)$, we have that 
 $$
 \|u(t)\| + \|f(t)\| \leq C e^{R't}, \quad \forall t \in [0, \infty), 
 $$
 for some $C$ positive. By the inversion formula of the Laplace transform, we have, for $Re(\lambda) > R'$, 
 $$
 \mL u(\lambda) = F(\lambda), \quad  \mL f(\lambda) = \lambda^\alpha F(\lambda).
 $$
 As 
 $$
\mL(\frac{t^{\alpha -1}}{\Gamma (\alpha)} ) (\lambda) = \lambda^{-\alpha}, \quad Re(\lambda) > 0, 
 $$
 we deduce that
 $$
 \mL v(\lambda) = F(\lambda), 
 $$
 so that $u = v$. 
  \end{proof}
 
 \begin{remark}\label{re2.2A}
 {\rm Suppose that $F$ fulfills the assumptions of Proposition \ref{pr7A} with $\alpha = 0$ and the (possible) exception of (c). Then $u \in B([0, T]; X)$. 
 }
 \end{remark}
  
Now we are able to define the Caputo derivative of order $\alpha$:

\begin{definition}\label{de2.2}
Let $\alpha \in \R^+$, $u \in C^{[\alpha]}([0, T]; X)$. Then the Caputo derivative of order $\alpha$ $\mbD^\alpha_X u$ exists if $u - \sum_{k < \alpha} \frac{t^k}{k!} u^{(k)}(0)$ belongs to $D(B^\alpha_X)$ and 
$$
\mbD^\alpha_X u := B_X^\alpha (u - \sum_{k < \alpha} \frac{t^k}{k!} u^{(k)}(0)). 
$$
\end{definition}

\begin{remark}
{\rm It is easy to see that, if $\alpha \in \N$, $\mbD^\alpha_X u$ exists if and only if $f \in C^\alpha([0, T]; X)$ and $\mbD^\alpha_X u = u^{(\alpha)}$. In any case, by Lemma \ref{le1.3}, $u \in C^\alpha([0, T]; X)$. 
 }
\end{remark}
 
 Now we introduce the following unbounded operator $A$: let $A(x, D_x)$ the partial differential operator introduced in $(A2)$. We set
\begin{equation}\label{eq2.6}
\left\{\begin{array}{l}
D(A) = \{u \in  \cap_{1 \leq p < \infty} (W^{2,p}(\Omega) \cap W^{1,p}_0(\Omega) : A(\cdot, D_x) u \in C(\overline \Omega)\}, \\ \\
Au = A(\cdot, D_x) u. 
\end{array}
\right. 
\end{equation}

\begin{proposition}\label{pr2.4}
Suppose that (A1)-(A2) hold. Then :

(I) there exist $\omega$ in $(0, (1 - \frac{\alpha}{2})\pi)$, $R$ and $C$ positive such that
$$
\{\lambda \in \C : |\lambda| \geq R, |Arg(\lambda)| \geq \omega\} \subseteq \rho(-A)
$$
and, if $\lambda$ is in this set, 
$$
\|(\lambda + A)^{-1}\|_{\mL(C(\overline \Omega)} \leq C |\lambda|^{-1}. 
$$

(II) As a consequence, there exists $\delta \geq 0$,  such that $\delta - A$ is a positive operator in $C(\overline \Omega)$ of type $\omega$. 

(III) If $\xi \in (0, 1) \setminus \{\frac{1}{2}\}$, $(C(\overline \Omega), D(A))_{\xi,\infty} = C^{2\xi}_0(\overline \Omega)$, with equivalent norms. 

(IV) If $\lambda \in \rho(-A)$ and $f\in C^\theta(\overline \Omega)$, $(\lambda + A)^{-1} f \in C^{2+\theta}(\overline \Omega)$. 

(V) There exists $C$ positive, such that, if $|\lambda| \geq R$ and $|Arg(\lambda)| \geq \omega$, 
\begin{equation}\label{eq2.8}
\|(\lambda + A)^{-1} f\|_{C^{2+\theta}(\overline \Omega)} + |\lambda| \|(\lambda + A)^{-1} f\|_{C^{\theta}(\overline \Omega)} \leq C( \|f\|_{C^{\theta}(\overline \Omega)} + |\lambda|^{\frac{\theta}{2}} \|\gamma f\|_{C(\partial \Omega)}). 
\end{equation}

\end{proposition}
\begin{proof} (I)  By compactness, there exist $\omega'$ in $(0, (1 - \frac{\alpha}{2})\pi)$, such that $|Arg(A(x, \xi)| \leq \omega'$ for any $(x,\xi)$ in $\overline \Omega \times (\R^n \setminus \{0\})$. Let $\omega \in (\omega', (1 - \frac{\alpha}{2})\pi)$.  So we have that, if  $|\psi| \in [\omega , \pi]$,
$$
e^{i \psi} \tau^2 - \sum_{|\alpha| = 2} a_\alpha(x) \xi^\alpha \neq 0 
$$
implying that $-e^{i\psi} D_t^2 + A(x, D_x)$ is properly elliptic in $\R \times \overline \Omega$ (see \cite{Ta1}, Chapter 3.7). If we have a properly elliptic operator with Dirichlet boundary conditions, the complementing condition is satisfied. So let $p \in (1, \infty)$. We set
$$
\left\{ \begin{array}{l}
D(A_p):= W^{2,p}(\Omega) \cap W_0^{1,p}(\Omega), \\ \\
A_p u = A(\cdot, D_x)u. 
\end{array}
\right. 
$$
$A_p$ is thought as an unbounded operator in $L^p(\Omega)$.  By Theorem 3.8.1 in \cite{Ta1}, there exists $R$ positive, such that, if $|\lambda| \geq R$ and  $|Arg(\lambda)| \geq \omega$, then $\lambda \in \rho(-A_p)$ and, for some $C_p$ positive, 
$$
\|(\lambda + A_p)^{-1}\|_{\mL(L^p(\Omega)} \leq C_p |\lambda|^{-1}. 
$$
Employing the method introduced in \cite{St1}, we deduce the claim. 

(II) follows immediately from (I). 

(III) can be proved with the argument in \cite{Gu2}, Theorem 3.6. 

(IV) Let $\epsilon \in [0, 1]$. We set 
$$
A_\epsilon(x,D_x):= (1 -\epsilon) \Delta_x + \epsilon A(x,D_x). 
$$
We observe that 
$$
|Arg((1 -\epsilon) |\xi|^2  + \epsilon A(x,\xi)| < (1 - \frac{\alpha}{2})\pi \quad \forall \epsilon \in [0, 1], \forall \xi \in \R^n \setminus \{0\}. 
$$
Employing a well known method due to Agmon (see, for example, \cite{Ta1}, Chapter 3.7) we can show the following a priori estimate: there exist $R', C$ positive, such that, if $\lambda| \geq R'$, $|Arg(\lambda)| \geq \omega$, $\epsilon \in [0, 1]$, $u \in C^{2+\theta}_0(\overline \Omega)$, 
$$
\|u\|_{C^{2+\theta}(\overline \Omega)} + |\lambda| \|u\|_{C^{\theta}(\overline \Omega)} \leq C (\|(\lambda + A_\epsilon)u\|_{C^{\theta}(\overline \Omega)} + |\lambda|^{\frac{\theta}{2}} \|(\lambda + A_\epsilon)u\|_{C(\overline \Omega)}). 
$$
If $\epsilon = 0$, it is well known that $(\lambda + A_0)^{-1} f \in C^{2+\theta}(\overline \Omega)$ in case $f \in C^{\theta}(\overline \Omega)$. So claim (IV) follows from the continuation method if $|\lambda| \geq R'$, $|Arg(\lambda)| \geq \omega$. The general case can be obtained fixing $\lambda_0$ in this set and recalling that, for any $\lambda$ in $\rho(A)$, 
$$
(\lambda + A)^{-1} = (\lambda_0 + A)^{-1} + (\lambda_0 - \lambda) (\lambda_0 + A)^{-1} (\lambda + A)^{-1} . 
$$

(V) can be obtained with the argument in \cite{Gu1}, Theorem 1.6. 

\end{proof}

\begin{remark}\label{re2.4}
{\rm By Proposition \ref{pr2.4} (I), if $\omega \in (0, (1 - \frac{\alpha}{2})\pi)$ is such that $|Arg(A(x, \xi))| < \omega$ for any $(x, \xi)$ in $\overline \Omega \times ( \R^n \setminus \{0\})$, for some $R$ positive
$$
\{\lambda \in \C : |\lambda| \geq R, |Arg(\lambda)| \leq \pi - \omega\} \subseteq \rho(A). 
$$
Moreover, if $\lambda$ is in this set
$$
\|(\lambda - A)^{-1}\|_{\mL(C(\overline \Omega)} \leq C|\lambda|^{-1}. 
$$
We observe that 
$$
\pi - \omega > \frac{\alpha \pi}{2}. 
$$

}
\end{remark}

Now we consider the abstract equation
\begin{equation}\label{eq2.7}
B^\alpha_X u(t) - Au(t) = f(t), \quad t \in [0, T], 
\end{equation}
with the following general conditions: 

\medskip

{\it (E) $X$ is a complex Banach space, $\alpha \in (0, 2)$, $A : D(A) (\subseteq X) \to X$ is an operator in $X$, such that, for some $\delta \geq 0$, $-A$ is positive of  type $\eta$ less than $(1 - \frac{\alpha}{2})\pi$. 
}

\medskip

We introduce the following 

\begin{definition}\label{de2.3}
Suppose that (E) holds. A strict solution of (\ref{eq2.7}) is an element $u$ of $D(B_X^\alpha) \cap C([0, T];$ $D(A))$ such that $B^\alpha_X u(t) - Au(t) = f(t)$ $\forall t \in [0, T]$. 
\end{definition}

It is convenient to introduce in the space $Y:= C([0, T]; X)$, the operator $\mA$, defined as follows: 
\begin{equation}
\left\{\begin{array}{l}
D(\mA) = C([0, T]; D(A)), \\ \\
(\mA u)(t) = Au(t), \quad t \in [0, T].  
\end{array}. 
\right. 
\end{equation}
and write (\ref{eq2.7}) in the form
\begin{equation}
(B_X^\alpha - \delta)u + (\delta - \mA)u = f. 
\end{equation}
We observe that $B_X^\alpha - \delta$ is positive of type $\frac{\alpha \pi}{2}$, $\delta - \mA$ is positive of type $\eta$, $\frac{\alpha \pi}{2} + \eta < \pi$. We observe also that $\rho(\mA) = \rho(A)$ and, $\forall \lambda$ in $\rho(A)$, $f \in C([0, T]; C(\overline \Omega))$, $t$ in $[0, T]$, 
$$
[(\lambda - \mA)^{-1}f](t) = (\lambda - A)^{-1} f(t). 
$$
Finally, if $\lambda \in \rho(B_X^\alpha)$, $\mu \in \rho(\mA)$, then
$$
(\lambda - B_X^\alpha)^{-1} (\mu - \mA)^{-1} = (\mu - \mA)^{-1} (\lambda - B_X^\alpha)^{-1}. 
$$
So we are in position to apply a slight generalization of the theory developed in \cite{DPGr1}, concerning sums of operators with commuting resolvents. This slight generalization can be found in \cite{DCGuLo1}, Theorem 2.2. Applying this theorem, we can deduce the following

\begin{proposition}\label{pr2.5}
Suppose that (E) holds. Then: 

(I) for any $f$ in $C([0, T]; X)$ (\ref{eq2.7}) has, at most, one strict solution $u$. 

(II) Such strict solution can be represented (if existing) in the form 
\begin{equation}\label{eq2.11}
u = Sf = \frac{1}{2\pi i} \int_{\Gamma (\pi - \eta',R)} (B^\alpha_X - \lambda)^{-1} (\lambda - \mA)^{-1} f d\lambda, 
\end{equation}
with  $\eta < \eta' < (1 - \frac{\alpha \pi}{2}) \pi$ and $R \in \R^+$ such that $\{\lambda \in \C : |\lambda| \geq R, |Arg(\lambda)| \leq \pi - \eta' \} \subseteq \rho (A)$. 
\end{proposition}

\begin{corollary}\label{co2.1}
Suppose that A1)-(A4) hold. Consider equation (\ref{eq2.7}),  in case $X = C(\overline \Omega)$ and $A$ is the operator defined in (\ref{eq2.6}). Then: 

(I)  for any $f$ in $C([0, T]; C(\overline \Omega))$ (\ref{eq2.7}) has, at most, one strict solution $u$.

(I) Such strict solution can be represented (if existing) in the form 

\begin{equation}\label{eq2.12}
u(t) = \frac{1}{2\pi i} \int_0^t ( \int_{\Gamma(\phi,r)} e^{\lambda (t-s)} (\lambda^\alpha - A)^{-1} d\lambda) f(s) ds, 
\end{equation}
with $\frac{\pi}{2} < \phi < \frac{\pi - \omega}{\alpha}$, $\omega$ as in Remark \ref{re2.4}, $r$ positive, such that $\{\lambda \in \C : |\lambda| \geq r^\alpha, |Arg(\lambda)| \leq \pi - \omega\} \subseteq \rho(A)$. 
\end{corollary}

\begin{proof} (I) We consider (\ref{eq2.11}), with $\eta' = \omega$. We fix $\phi$  as in the statement and $r$ positive such that $r^\alpha > R$. Then it follows from Lemma \ref{le1.2A} (d) that, for any $\lambda$ in $\Gamma (\pi - \omega,R)$, one has
$$
(\lambda - B_X^\alpha)^{-1}f(t) =  \frac{1}{2\pi i}  \int_0^t ( \int_{\Gamma(\phi,r)} \frac{e^{\mu (t-s)}}{\lambda - \mu^\alpha} d\mu) f(s) ds.
$$
So, from (\ref{eq2.11}) we deduce, applying Cauchy's integral formula, that, for any $t$ in $[0, T]$, 
$$
\begin{array}{c}
u(t) = \frac{1}{(2\pi i)^2} \int_{\Gamma(\pi - \omega,R)} (\int_0^t (\int_{\Gamma(\phi,r)} \frac{e^{\mu(t-s)}}{\mu^\alpha - \lambda } d\mu) (\lambda - A)^{-1} f(s) ds ) d\lambda \\ \\
=  \frac{1}{2\pi i}\int_0^t (\int_{\Gamma(\phi,r)} e^{\mu(t-s)} ( \frac{1}{2\pi i}  \int_{\Gamma(\pi - \omega,R)} (\mu^\alpha - \lambda)^{-1} (\lambda - A)^{-1} d\lambda) d\mu) f(s) ds \\ \\
= \frac{1}{2\pi i}\int_0^t (\int_{\Gamma(\phi,r)} e^{\mu(t-s)} (\mu^\alpha - A)^{-1} d\mu) f(s) ds. 
\end{array}
$$
The second identity is justified by the estimate
$$
\begin{array}{c}
\int_{\Gamma(\pi - \omega,R)} (\int_0^t (\int_{\Gamma(\phi,r)} \frac{e^{Re(\mu)(t-s)}}{|\mu^\alpha - \lambda|} |d\mu| ) \| (\lambda - A)^{-1} f(s)\| ds ) |d\lambda| \\ \\
\leq C_1 \int_{\Gamma(\pi - \omega,R)} (\int_{\Gamma(\phi,r)}  \min\{t, |Re(\mu)|^{-1}\} |\lambda|^{-1} (|\mu|^\alpha + |\lambda|)^{-1} |d\mu|) |d\lambda| \\ \\
\leq C_2 \int_{\Gamma(\phi,r)} \min\{t, |Re(\mu)|^{-1}\} |\mu|^{-\alpha} ln(|\mu| + 1) |d\mu| < \infty. 
\end{array}
$$
\end{proof}

We pass to consider the abstract system (\ref{eq1.3}). 

\begin{definition} Let $X$ be a complex Banach space, $\alpha \in \R^+$, $A$ a closed operator in $X$, $f \in C([0, T]; X)$, $u_k \in X$ for each $k \in \N_0$, $k < \alpha$. A strict solution $u$ of (\ref{eq1.3}) is an element of $C([0, T]; D(A))$, such that 
$\mathbb{D}^\alpha_{X} u$ is defined and all the conditions in (\ref{eq1.3}) are satisfied pointwise. 
\end{definition}

The two following results of maximal regularity hold: 

\begin{theorem} \label{th2.1A} Suppose that (E) holds. Let $\beta \in (0, \min\{1, \alpha\})$. Then the following conditions are necessary and sufficient in order that (\ref{eq1.3}) has a unique strict solution $u$, with $\mathbb{D}^\alpha_{X} u$ and $Au$ (that is, $\mA u$) belonging to $C^\beta ([0, T]; X)$:

(a) $f \in C^\beta ([0, T]; X)$; 

(b) $u_0 \in D(A)$; 

(c) $Au_0 + f(0) \in (X, D(A))_{\beta/\alpha , \infty}$; 

(d) if $\alpha > 1$, $u_1 \in (X, D(A))_{1- \frac{1-\beta}{\alpha}, \infty}$. 

\end{theorem}

\begin{theorem} \label{th2.1B} Suppose that (E) holds. Let $\beta \in (0, 1)$, $\alpha \beta < 1$. Then the following conditions are necessary and sufficient in order that (\ref{eq1.3}) has a unique strict solution $u$, with $\mathbb{D}^\alpha_{X} u$ and $Au$ belonging to $C ([0, T]; X) \cap B([0, T]; (X, D(A))_{\beta,\infty})$:

(a) $f \in C([0, T]; X) \cap B([0, T]; (X, D(A))_{\beta,\infty})$; 

(b) $u_0 \in D(A)$, $Au_0 \in (X, D(A))_{\beta,\infty}$; 

(c) if $\alpha > 1$, $u_1 \in (X, D(A))_{\beta + 1 - \frac{1}{\alpha}, \infty}$. 

\end{theorem}

As we already mentioned, the case $\alpha \in (0, 1]$ is treated in \cite{ClGrLo1}. Concerning the case $\alpha \in (1, 2)$, the sufficiency of the conditions (a)-(d) and (a)-(c) to get the conclusion is proved in \cite{ClGrLo2}. Their necessity is shown in \cite{Gu3}. 

Applying Theorems \ref{th2.1A}-\ref{th2.1B} in the case that $A$ is the operator defined in (\ref{eq2.6}), we deduce, on account of Proposition \ref{pr2.4}: 

\begin{corollary} \label{co2.2} Suppose that (A1)-(A2) are fulfilled. We consider system \ref{eq1.3} in the case $X = C(\overline \Omega)$, and $A$ as in (\ref{eq2.6}). Let $\alpha \in (0, 2)$, $\beta \in (0, \min\{1, \alpha\})$. Then the following conditions are ne\-cessary and sufficient in order that (\ref{eq1.3}) has a unique strict solution $u$, with $\mathbb{D}^\alpha_{X} u$ and $Au$ belonging to $C^\beta ([0, T]; C(\overline \Omega))$:

(a) $f \in C^\beta ([0, T]; C(\overline \Omega))$; 

(b) $u_0 \in D(A)$; 

(c) $A(\cdot,D_x) u_0 + f(0) \in (C(\overline \Omega), D(A))_{\beta/\alpha , \infty}$ ($A(\cdot,D_x) u_0 + f(0) \in C_0^{2\beta/\alpha} (\overline \Omega)$ in case $\beta \neq \frac{\alpha}{2}$); 

(c) if $\alpha > 1$, $u_1 \in (C(\overline \Omega), D(A))_{1- \frac{1-\beta}{\alpha}, \infty}$ ($u_1 \in C^{2(1- \frac{1-\beta}{\alpha})}_0 (\overline \Omega)$ in case $\beta \neq1 - \frac{\alpha}{2}$) . 

\end{corollary}

\begin{corollary} \label{co2.3} Suppose that (A1)-(A2) are fulfilled. We consider system \ref{eq1.3} in the case $X = C(\overline \Omega)$, and $A$ as in (\ref{eq2.6}). Let $\alpha \in (0, 2)$, $\alpha \theta < 2$. Then the following conditions are necessary and sufficient in order that (\ref{eq1.3}) has a unique strict solution $u$, with $\mathbb{D}^\alpha_{X} u$ and $Au$ belonging to $C([0, T]; C(\overline \Omega)) \cap B([0, T]; C^\theta_0(\overline \Omega))$: 

(a) $f \in C ([0, T]; C(\overline \Omega)) \cap B([0, T]; C^\theta_0(\overline \Omega))$; 

(b) $u_0 \in C^{2+\theta}_0(\overline \Omega)$, $A(\cdot,D_x) u_0 \in C^{\theta}_0(\overline \Omega)$; 

(c) if $\alpha > 1$, $u_1 \in (C(\overline \Omega), D(A))_{\frac{\theta}{2} + 1 - \frac{1}{\alpha}, \infty}$ ($u_1 \in C^{\theta + 2 - \frac{2}{\alpha}}_0 (\overline \Omega)$ in case $\alpha \neq \frac{2}{\theta +1}$). 

\end{corollary}

\section{Proof of Theorem \ref{th1.1}} \label{se3}

We begin with some preliminary results. 

\setcounter{equation}{0}

\begin{lemma} \label{le1.4} Suppose that (A1)-(A4) hold. Let $u \in C([0, T]; C(\overline \Omega))$ satisfy (B1)-(B2). Then:
 
 (I) if $\alpha \in (0, 2) \setminus \{1\}$, $u \in C^\alpha([0, T]; C^\theta(\overline \Omega))$; 
 
 (II) if $\alpha = 1$, $u \in Lip([0, T]; C^\theta(\overline \Omega))$; 
 
 (III) in any case, $u \in C^{\frac{\alpha \theta}{2}}([0, T]; C^2(\overline \Omega))$; 
 
 (IV) $u(0) \in C^{2+\theta}(\overline \Omega)$; 
 
 (V) if $\alpha \in (1, 2)$,  $D_tu(0) \in C^{\theta + 2(1 - 1/\alpha)}(\overline \Omega)$. 
 \end{lemma}
 
 \begin{proof} (IV) is obvious.
 
 We show (I). We set $v(t):= u(t) - \sum_{k < \alpha} t^k D_t^{(k)}(0)$. Then, by Lemma \ref{le1.3}, with $X = C(\overline \Omega)$, $X_1 = C^\theta(\overline \Omega)$, $v \in C^\alpha ([0, T]; C^\theta(\overline \Omega))$. 
 By (IV), in case $\alpha \in (0, 1)$, we obtain the assertion,  because
 $$u(t) = v(t) + u(0).$$
 Assume that $\alpha \in (1, 2)$. Then, by difference,
 $tD_t u(0) = u(t) - u(0) - v(t) \in C^\theta(\overline \Omega)$ for any $t \in [0, T]$. We deduce that necessarily $D_t u(0) \in C^\theta(\overline \Omega)$. So the conclusion follows from
 $$u(t) = v(t) + u(0) + t D_tu(0).$$
 The proof of (II) is similar.
 
 We show (V). By Theorem 3.2 in \cite{DCGuLo1}, from $u \in C^\alpha([0, T]; C^\theta(\overline \Omega)) \cap B([0, T]; C^{2+\theta}(\overline \Omega))$, we deduce that $D_t u$ is bounded with values in the interpolation space
 $$(C^\theta(\overline \Omega, C^{2+\theta}(\overline \Omega))_{1 - \frac{1}{\alpha}, \infty} =  C^{\theta + 2(1 - 1/\alpha)}(\overline \Omega),$$
 by Lemma \ref{le2.1} (II). From this we deduce also that 
 \begin{equation}\label{eq3.1A}
 u \in Lip([0, T]; C^{\theta + 2(1 - 1/\alpha)}(\overline \Omega)). 
 \end{equation}
 
 We show (III). We recall that $C^2(\overline \Omega)) \in J_{1-\theta/2}(C^\theta(\overline \Omega)), C^{2+\theta}(\overline \Omega))$ (by Lemma \ref{le2.1}(I)). So, in case $\alpha \in (0, 1]$, from (I)-(II) we deduce, $\forall t, s \in [0, T]$, for some positive constant $C$,
 $$
 \begin{array}{c}
 \|u(t) - u(s)\|_{C^2(\overline \Omega))} \leq C \|u(t) - u(s)\|_{C^\theta(\overline \Omega))}^{\frac{ \theta}{2}} \|u(t) - u(s)\|_{C^{2+\theta}(\overline \Omega))}^{1-\frac{ \theta}{2}} \\ \\
  \leq  2^{1-\theta/2} C \|u\|_{C^\alpha([0, T]; C^\theta(\overline \Omega))}^{\theta/2}   \|u\|_{B([0, T]; C^{2+\theta}(\overline \Omega))}^{1-\theta/2}(t-s)^{\frac{ \alpha \theta}{2}}. 
 \end{array}
 $$
 Suppose that $\alpha \in (1, 2)$. Then, (\ref{eq3.1A}) holds. Observe that $\theta + 2(1 - 1/\alpha) < 2$, because $\alpha \theta < 2$. So the conclusion follows from the fact that
 $$
 C^2(\overline \Omega) \in J_{1- \alpha \theta/2} (C^{\theta + 2(1 - 1/\alpha)}(\overline \Omega), C^{\theta + 2}(\overline \Omega)). 
 $$
 \end{proof}

 \begin{lemma}\label{le3.2}
 Suppose that the assumptions (A1)-(A4) are satisfied. Then the conditions (I)-(VI) in the statement of Theorem \ref{th1.1} are necessary, in order that there exists a solution $u$ fulfilling (B1)-(B2). 
 \end{lemma}
 
 \begin{proof} (I) is obviously necessary. 
 
 The necessity of (II) follows from Lemma \ref{le1.4} (IV)-(V). 
 
 The necessity of (IV) is clear. 
 
 We show that (III) is necessary. First, as $u \in C([0, T]; C^2(\overline \Omega)) \cap B([0, T]; C^{2+\theta}(\overline \Omega))$, necessarily $g = \gamma u \in C([0, T]; C^2(\partial \Omega)) \cap B([0, T]; C^{2+\theta}(\partial \Omega))$
 Next, we set $h:= \mbD^\alpha_{C(\overline \Omega)} u$. Then $h \in C([0, T]; C(\overline \Omega)) \cap B([0, T]; C^\theta (\overline \Omega)$ and 
 $$
 u(t) = \sum_{k < \alpha} t^k u_k + \frac{1}{\Gamma(\alpha)} \int_0^t (t-s)^{\alpha -1} h(s) ds, \quad t \in [0, T]. 
 $$
 From (IV) we obtain
  $$
 g(t) = \sum_{k < \alpha} t^k g^{(k)}(0) + \frac{1}{\Gamma(\alpha)} \int_0^t (t-s)^{\alpha -1} \gamma h(s) ds, \quad t \in [0, T], 
 $$
 implying that $\mbD^\alpha_{C(\partial \Omega)} g$ is defined  and $\mbD^\alpha_{C(\partial \Omega)} g = \gamma h = \gamma \mbD^\alpha_{C(\overline\Omega)} u$. 
 
 So $\mbD^\alpha_{C(\partial \Omega)} g$ has to belong to $C([0, T]; C(\partial \Omega)) \cap B([0, T]; C^\theta(\partial \Omega))$. 
 
 We show that (V) is necessary. We have
 $$
 \gamma f - \mbD^\alpha_{C(\partial \Omega)} g = \gamma (f - \mbD^\alpha_{C(\overline \Omega)}  u) = - \gamma (A(\cdot,D_x)u). 
 $$
 By Lemma \ref{le1.4} (III), $A(\cdot,D_x)u \in C^{\alpha \theta/2}([0, T]; C(\overline \Omega))$, so that $ \gamma (A(\cdot,D_x)u) \in C^{\alpha \theta/2}([0, T]; C(\partial \Omega))$. 
 
 Finally, (VI) follows from
 $$
 \gamma[A(\cdot,D_x) u_0 + f(0)] = \gamma (\mbD^\alpha_{C(\overline \Omega)} u)(0) =  \mbD^\alpha_{C(\partial \Omega)} g(0). 
 $$
 \end{proof}
 
 It remains to prove that the assumptions (I)-(VI) of Theorem \ref{th1.1} are also sufficient. To this aim, we begin to consider the case $u_0 = u_1 = 0$, $g \equiv 0$. So we consider the equation
 \begin{equation}\label{eq3.2}
 B^\alpha_{C(\overline \Omega)} u(t) = Au(t) + f(t), \quad t \in [0, T], 
 \end{equation}
 ($A$ is the operator defined in (\ref{eq2.6})), with the following conditions: 
 
 \medskip
 
 {\it (C1) $f \in C([0, T]; C(\overline \Omega)) \cap B([0, T]; C^\theta(\overline \Omega))$.
 
 \medskip
 
 (C2) $\gamma f  \in C^{\frac{\alpha \theta}{2}} ([0, T]; C(\partial \Omega))$. 
 
 \medskip
 
 (C3) $\gamma[f(0)] = 0$. 
 }
 
 \medskip
 
 By Corollary \ref{co2.1}, the unique possible solution of (\ref{eq3.2}) is 
 \begin{equation}\label{eq3.3}
 u(t) = \int_0^t T(t-s) f(s) ds, 
 \end{equation}
 with
 \begin{equation}\label{eq3.4}
 T(t) = \frac{1}{2\pi i} \int_{\Gamma (\phi,r)} e^{\lambda t} (\lambda^\alpha - A)^{-1} d\lambda, 
 \end{equation}
 with $\phi$ in $(\frac{\pi}{2}, \frac{\pi - \omega}{\alpha}]$ and $r \geq R^{1/\alpha}$, $\omega$ and $R$ as in Remark \ref{re2.4}. Then we have: 
 
 \begin{lemma}\label{le3.3A}
 Suppose that (C1)-(C3) are satisfied. Then the function $u$ given by (\ref{eq3.3}) is a strict solution of (\ref{eq3.2}). 
 \end{lemma}
 
 \begin{proof}
Let $R$ be the operator introduced in Lemma \ref{le2.1} (IV).  We observe that, for any $t$ in $[0, T]$,
 $$f(t) = (f(t) - R\gamma f(t)) + R \gamma f(t).$$
 $f - R\gamma f$ belongs to $C([0, T]; C(\overline \Omega)) \cap B([0, T]; C^\theta_0(\overline \Omega))$, $R\gamma f$ belongs to $C^{\frac{\alpha \theta}{2}} ([0, T]; C(\overline \Omega))$ and $R\gamma f(0) = 0$. So the conclusion follows from Corollaries \ref{co2.2})-\ref{co2.3}). 

 \end{proof}
 
 \begin{lemma}\label{le3.4}
 Let $\omega$ and $R$ be as in Proposition \ref{pr2.4}. Let $\xi \in [\theta, 2+\theta]$. Then there exists $C(\xi)$ positive such that, $\forall \lambda \in \C$, with $|\lambda| \geq R$, $|Arg(\lambda)| \geq \omega$, $\forall f \in C^\theta (\overline
 \Omega)$, 
 $$
\|(\lambda + A)^{-1} f\|_{C^{\xi}(\overline \Omega)} \leq C(\xi) |\lambda|^{\frac{\xi - \theta}{2} -1}(\|f\|_{C^{\theta}(\overline \Omega)} + |\lambda|^{\frac{\theta}{2}} \|\gamma f\|_{C(\partial \Omega)}). 
$$
 \end{lemma}
 
 \begin{proof} The case $\xi \in \{\theta, 2+\theta\}$ follows from Proposition \ref{pr2.4} (V). The case $\xi \in (\theta, 2+\theta)$ follows from the foregoing and Lemma \ref{le2.1} (I). 
 \end{proof}
 
 As a consequence, we obtain the following
 
 \begin{lemma}\label{le3.5}
 Let us consider the family of operators $(T(t))_{t > 0}$, introduced in (\ref{eq3.4}). Then $T \in C(\R^+; $ $\mL(C^\theta(\overline \Omega), C^{2+\theta}(\overline \Omega)))$. Moreover, for any $\xi$ in $[\theta, 2+\theta]$ there exists $C(\xi)$ positive, such that $\forall f \in C^\theta(\overline \Omega)$, $t$ in $(0, T]$, 
 $$
 \|T(t) f\|_{C^\xi(\overline \Omega)} \leq C(\xi) t^{\alpha(1 - \frac{\xi - \theta}{2})- 1}(\|f\|_{C^{\theta}(\overline \Omega)} + t^{-\frac{\alpha \theta}{2}}\|\gamma f\|_{C(\partial \Omega)}). 
 $$
 \end{lemma}
 
 \begin{proof} Again, we fix $\phi$ in $(\frac{\pi}{2}, \frac{\pi - \omega}{\alpha}]$ and $r \geq R^{1/\alpha}T$, with $\omega$ and $R$ as in Remark \ref{re2.4}. 
 Then we have
 $$
 T(t) f = \frac{1}{2\pi i} \int_{\Gamma (\phi,\frac{r}{t})} e^{\lambda t} (\lambda^\alpha - A)^{-1} f d\lambda = \frac{t^{-1}}{2\pi i} \int_{\Gamma (\phi,r)} e^{\lambda} (t^{-\alpha} \lambda^\alpha - A)^{-1} f d\lambda
 $$
 so that
 $$
 \|T(t) f\|_{C^\xi(\overline \Omega)} \leq C_1(\xi) t^{-1} \int_{\Gamma (\phi,r)} e^{Re\mu}  |t^{-\alpha} \mu^\alpha|^{\frac{\xi - \theta}{2} -1}(\|f\|_{C^{\theta}(\overline \Omega)} + |t^{-\alpha} \mu^\alpha|^{\frac{\theta}{2}} \|\gamma f\|_{C(\partial \Omega)}) |d\mu|, $$
 which implies the statement. 
  \end{proof}
  
  \begin{lemma}\label{le3.6}
  If $t \in \R^+$, we set 
  \begin{equation}\label{eq3.5}
  T_1(t):= \int_0^t T(s) ds. 
  \end{equation}
  Then $T_1 \in C(\R^+; \mL(C^\theta(\overline \Omega), C^{2+\theta}(\overline \Omega))$. Moreover, for any $\xi$ in $[\theta, 2+\theta]$ there exists $C(\xi)$ positive, such that $\forall f \in C^\theta(\overline \Omega)$, $t$ in $(0, T]$, 
 $$
 \|T_1(t) f\|_{C^\xi(\overline \Omega)} \leq C(\xi) t^{\alpha(1 - \frac{\xi - \theta}{2})}(\|f\|_{C^{\theta}(\overline \Omega)} + t^{-\frac{\alpha \theta}{2}}\|\gamma f\|_{C(\partial \Omega)}). 
 $$
  \end{lemma}
   \begin{proof} We start by observing that, in force of Lemma \ref{le3.5}, the integral in (\ref{eq3.5}) converges in $\mL(C^\theta(\overline \Omega),$ $C^\xi(\overline \Omega))$, for any $\xi$ in $[\theta, 2)$. In general, it is easily seen that
   $$
   T_1(t) = \frac{1}{2\pi i} \int_{\Gamma (\phi,r)} e^{\lambda t} \lambda^{-1} (\lambda^\alpha - A)^{-1} d\lambda. 
   $$
In fact, the second term vanishes for $t = 0$ and has derivative $T(t)$ for $t$ positive.   So the assertion can be obtained with the same method of Lemma \ref{le3.5}. 
  \end{proof}
  
  \begin{lemma} \label{le3.7} Let $\alpha \in (0, 2)$, let $\phi \in (\frac{\pi}{2}, \frac{\pi}{\alpha})$, $r, \xi$ positive, such that $r^\alpha < \xi$. Then, for any $t$ positive
  $$
 h(t,\xi):=  \frac{1}{2\pi i} \int_{\Gamma (\phi,r)} \frac{e^{\lambda t}}{\lambda^\alpha - \xi} d\lambda  = \frac{1}{\pi} \int_0^\infty \frac{e^{-t \tau} \tau^\alpha \sin(\alpha \pi)} {\tau^{2\alpha} - 2\xi \cos(\alpha \pi) \tau^\alpha + \xi^2} d\tau. 
  $$
  \end{lemma}
  \begin{proof} By standard properties of holomorphic functions, we have 
  $$
  \begin{array}{c}
  h(t,\xi) = \frac{1}{2\pi i} (\int_0^\infty \frac{e^{e^{i\phi} \tau t}}{e^{i \alpha \phi} \tau^\alpha - \xi} e^{i\phi} d\tau - \int_0^\infty \frac{e^{e^{-i\phi} \tau t}}{e^{-i \alpha \phi} \tau^\alpha - \xi} e^{-i\phi} d\tau) \\ \\
  = \frac{1}{2\pi i} (\int_0^\infty \frac{e^{e^{i\beta} \tau t}}{e^{i \alpha \beta} \tau^\alpha - \xi} e^{i\beta} d\tau - \int_0^\infty \frac{e^{e^{-i\beta} \tau t}}{e^{-i \alpha \beta} \tau^\alpha - \xi} e^{-i\beta} d\tau)
  \end{array}
  $$
  for any $\beta \in [\phi, \pi]$, as $0 < \alpha \beta < 2\pi$ and $0 > - \alpha \beta > - 2\pi$. Taking $\beta = \pi$, we obtain
  $$
  h(t,\xi) = \frac{1}{2\pi i} (- \int_0^\infty \frac{e^{-\tau t}}{e^{i \alpha \pi} \tau^\alpha - \xi} d\tau + \int_0^\infty \frac{e^{-\tau t}}{e^{-i \alpha \pi} \tau^\alpha - \xi} d\tau),
  $$
  from which the assertion follows. 
\end{proof}

\begin{remark}
{\rm In case $\alpha = 1$, we have $h(t,\xi) = 0$. }
\end{remark}

In the following we shall use several times the  elementary

\begin{lemma}\label{le3.8}
Let $a, b, c, d$ be real numbers, such that $a > 0$, $-1 < b < c < b+d$. Then there exists $C$ positive, depending on $a, b, c, d$, such that, for any $\xi > 0$, 
$$
\int_{\R^+ \times \R^+} e^{-at\tau} \frac{t^b \tau^c}{\tau^d + \xi} dt d\tau = C \xi^{\frac{c-b}{d} -1}.
$$
\end{lemma}
We omit the simple proof. 

Now we are in position to show the following

\begin{proposition}\label{pr3.1} Suppose that (A1)-(A4) and (C1)-(C3) are fulfilled. Let $u$ be as in (\ref{eq3.3}). Then $u$ is a strict solution of (\ref{eq3.2}). Moreover, $B^\alpha_{C(\overline \Omega)} u \in B([0, T]; C^\theta (\overline \Omega))$ and $u \in B([0, T]; C^{2+\theta}(\overline \Omega))$. 
\end{proposition}

\begin{proof} The fact that $u$ is a strict solution has been proved in Lemma \ref{le3.3A}. In order to show the remainng part of the assertion, it suffices to show that $u \in C([0, T]; C^2(\overline \Omega))$ and that $Au \in B([0, T]; C^\theta(\overline \Omega))$, in force of the inequality
$$
\|u\|_{C^{2+\theta}_0(\overline \Omega)} \leq C(\|u\|_{C^{2}(\overline \Omega)} + \|Au\|_{C^{\theta}(\overline \Omega)}), 
$$
which is valid for some $C$ positive independent of $u$, by Proposition \ref{pr2.4}. 

To this aim, we begin by observing that
$$
u(t) = \int_0^t T(t-s) [f(s) - f(t)] ds + T_1(t) f(t) := u_1(t) + u_2(t). 
$$
Then we have, by Lemma \ref{le3.6}, 
$$
\|u_2(t)\|_{C^2(\overline \Omega)} \leq C_1 t^{\frac{\alpha \theta}{2}} (\|f(t)\|_{C^{\theta}(\overline \Omega)} + t^{- \frac{\alpha \theta}{2}} \|\gamma f(t)\|_{C(\partial \Omega)}) \leq C_2 t^{\frac{\alpha \theta}{2}}, 
$$
so that $u_2 \in C([0, T]; C^2(\overline \Omega))$. Moreover, again by Lemma \ref{le3.6}, 
$$
\|Au_2(t)\|_{C^\theta(\overline \Omega)} \leq C_3 \|u_2(t)\|_{C^{2+\theta}(\overline \Omega)} \leq C_4 (\|f(t)\|_{C^{\theta}(\overline \Omega)} + t^{- \frac{\alpha \theta}{2}} \|\gamma f(t)\|_{C(\partial \Omega)}) \leq C_5. 
$$
So $u_2 \in B([0, T]; C^{2+\theta}(\overline \Omega))$. 

Now we consider $u_1$. By Lemma \ref{le3.5}, 
$$
\begin{array}{ll}
\|u_1(t)\|_{C^2(\overline \Omega)} &  \leq C(2) \int_0^t (t-s)^{\frac{\alpha \theta}{2} -1} (2 \|f\|_{B([0, T]; C^\theta(\overline \Omega))} + \|\gamma f\|_{C^{\frac{\alpha \theta}{2}}([0, T]; C(\overline \Omega)}) ds \\ \\
&  \leq C_6 t^{\frac{\alpha \theta}{2}} (\|f\|_{B([0, T]; C^\theta(\overline \Omega))} + \|\gamma f\|_{C^{\frac{\alpha \theta}{2}}([0, T]; C(\overline \Omega)}). 
\end{array}
$$
So $u_1 \in C([0, T]; C^2(\overline \Omega))$. It remains to estimate $\|Au_1(t)\|_{C^\theta(\overline \Omega)}$. By Proposition \ref{pr2.1A} (III) and Proposition \ref{pr2.4} (III), 
$$
C^{\theta}_0(\overline \Omega) = \{f \in C(\overline \Omega) : \sup_{\xi \geq 2R} \xi^{\frac{\theta}{2}} \|A (\xi - A)^{-1} f\|_{C(\overline \Omega)} < \infty\},
$$
with $R$ as in Remark \ref{re2.4}. Moreover, the norm
$$
f \to \max\{\|f\|_{C(\overline \Omega)}, \sup_{\xi \geq 2R} \xi^{\frac{\theta}{2}} \|A (\xi - A)^{-1} f\|_{C(\overline \Omega)}\}
$$
with $f \in C^\theta_0(\overline \Omega)$, is equivalent to $\|\cdot\|_{C^\theta_0(\overline \Omega)}$. So, in order to complete the proof, we can show that there exists $C$ positive, such that, for any $t$ in $(0, T]$, for any $\xi$ in 
$[2R, \infty)$, 
\begin{equation} \label{eq3.6}
\|A (\xi - A)^{-1} Au_1(t)\|_{C(\overline \Omega)} \leq C \xi^{-\frac{\theta}{2}}. 
\end{equation}
We put
\begin{equation}\label{eq3.7}
\Gamma:= \Gamma (\frac{\pi - \omega}{\alpha}, R^{1/\alpha}). 
\end{equation}
Ler $\xi \in [2R, \infty)$. Then we have, by the resolvent identity, 
$$
\begin{array}{c}
A(\xi - A)^{-1} A T(t) = \frac{1}{2\pi } \int_\Gamma e^{\lambda t} A(\xi - A)^{-1} A (\lambda^\alpha - A)^{-1} d\lambda = \\ \\
\frac{1}{2\pi i} \int_\Gamma e^{\lambda t} A(\xi - A)^{-1} [- 1 + \lambda^\alpha (\lambda^\alpha - A)^{-1}] d\lambda = \frac{1}{2\pi i} \int_\Gamma e^{\lambda t} \lambda^\alpha A(\xi - A)^{-1}  (\lambda^\alpha - A)^{-1} d\lambda \\ \\
= \frac{1}{2\pi i} \int_\Gamma e^{\lambda t} \frac{\lambda^\alpha}{\xi - \lambda^\alpha} A(\lambda^\alpha - A)^{-1} d\lambda + \frac{1}{2\pi i} \int_\Gamma e^{\lambda t} \frac{\lambda^\alpha}{\lambda^\alpha - \xi}  d\lambda \hspace{.05in} A(\xi - A)^{-1} \\ \\
:= K_1(t,\xi) + K_2(t,\xi), 
\end{array}
$$
so that 
$$
A(\xi -A)^{-1} A u_1(t) = \int_0^t K_1(t-s, \xi) [f(s) - f(t)] ds + \int_0^t K_2(t-s, \xi) [f(s) - f(t)] ds. 
$$
We have
$$
\begin{array}{c}
\|\int_0^t K_1(t-s, \xi) [f(s) - f(t)] ds\|_{C(\overline \Omega)} \\ \\
\leq \frac{1}{2\pi}  \int_0^t ( \int_\Gamma  e^{(t -s) Re(\lambda)} \frac{|\lambda|^\alpha}{|\xi - \lambda^\alpha|} \|A(\lambda^\alpha - A)^{-1} [f(s) - f(t)]\|_{C(\overline \Omega)} |d\lambda|) 
ds
\end{array}
$$
We indicate with $\Gamma_1$, $\Gamma_2$, $\Gamma$ piecewise regular paths describing, respectively, 
$$\{\lambda \in \C : |\lambda| = R^{1/\alpha}, |Arg(\lambda)| \leq \frac{\pi - \omega}{\alpha}\}, $$
$$\{\lambda \in \C : |\lambda| \geq R^{1/\alpha}, Arg(\lambda) = \frac{\pi - \omega}{\alpha}\}, $$
$$\{\lambda \in \C : |\lambda| \geq R^{1/\alpha}, Arg(\lambda) = -\frac{\pi - \omega}{\alpha}\}.  $$
Then, clearly, 
$$\int_0^t ( \int_{\Gamma_1}  e^{(t -s)Re(\lambda)} \frac{|\lambda|^\alpha}{|\xi - \lambda^\alpha|} \|A(\lambda^\alpha - A)^{-1} [f(s) - f(t)]\|_{C(\overline \Omega)} |d\lambda|) ds \leq C_1 \xi^{-1} \|f\|_{C([0, T]; C(\overline \Omega))}, $$
$$
\begin{array}{c}
\int_0^t ( \int_{\Gamma_2}  e^{(t-s) Re(\lambda)} \frac{|\lambda|^\alpha}{|\xi - \lambda^\alpha|} \|A(\lambda^\alpha - A)^{-1} [f(s) - f(t)]\|_{C(\overline \Omega)} |d\lambda|) ds \\ \\
+ \int_0^t ( \int_{\Gamma_3}  e^{(t-s) Re(\lambda)} \frac{|\lambda|^\alpha}{|\xi - \lambda^\alpha|} \|A(\lambda^\alpha - A)^{-1} [f(s) - f(t)]\|_{C(\overline \Omega)} |d\lambda|) ds \\ \\
:= I + J. 
\end{array} 
$$
By Lemma \ref{le3.4}, we have
$$
\begin{array}{c}
\|A(\lambda^\alpha - A)^{-1} [f(s) - f(t)]\|_{C(\overline \Omega)} \leq C_2 \|(\lambda^\alpha - A)^{-1} [f(s) - f(t)]\|_{C^2(\overline \Omega)} \\  \\
\leq C_3 (|\lambda|^{-\frac{\alpha \theta}{2}} \|f\|_{B([0, T]; C^\theta(\overline \Omega))} + (t-s)^{\frac{\alpha \theta}{2}} \|\gamma f\|_{C^{\frac{\alpha \theta}{2}}([0, T]; C(\partial \Omega))}), 
\end{array}
$$
so that 
$$
\begin{array}{c}
I + J \leq C_4 (\int_{\R^+ \times \R^+} e^{- |\cos(\frac{\pi - \omega}{\alpha})|t \tau} \frac{\tau^{\alpha - \frac{\alpha \theta}{2}}}{\tau^\alpha + \xi} dt d\tau \|f\|_{B([0, T]; C^\theta(\overline \Omega))} \\ \\
+ \int_{\R^+ \times \R^+} e^{- |\cos(\frac{\pi - \omega}{\alpha})|t \tau} \frac{t^\frac{\alpha \theta}{2} \tau^\alpha }{\tau^\alpha + \xi} dt d\tau  \|\gamma f\|_{C^{\frac{\alpha \theta}{2}}([0, T]; C(\partial \Omega))}) \\ \\
= C_5 \xi^{-\frac{\theta}{2}}(\|f\|_{B([0, T]; C^\theta(\overline \Omega))} + \|\gamma f\|_{C^{\frac{\alpha \theta}{2}}([0, T]; C(\partial \Omega))}),
\end{array}
$$
in force of Lemma  \ref{le3.8}. 

Finally, we observe that 
$$
\frac{1}{2\pi i} \int_\Gamma e^{\lambda t} \frac{\lambda^\alpha}{\lambda^\alpha - \xi}  d\lambda = \xi h(t,\xi), 
$$
so that 
$$
\int_0^t K_2(t-s, \xi) [f(s) - f(t)] ds = \int_0^t h(t-s,\xi) \xi A(\xi -A)^{-1} [f(s) - f(t)] ds.
$$
So, by Lemmata \ref{le3.4}, \ref{le3.7}, \ref{le3.8},  we have
$$
\begin{array}{c}
\|\int_0^t K_2(t-s, \xi) [f(s) - f(t)] ds\|_{C(\overline \Omega)} \\ \\
 \leq C_6 \xi \int_0^t (\int_{\R^+} e^{-(t-s) \tau} \frac{\tau^\alpha}{\tau^{2\alpha} + \xi^2} d\tau) \|A(\xi -A)^{-1} [f(s) - f(t)]\|_{C(\overline \Omega)} ds \\ \\
 \leq C_7 \xi (\int_{\R^+ \times \R^+} e^{-t \tau} \frac{\tau^\alpha}{\tau^{2\alpha} + \xi^2}  dt d\tau \xi^{-\theta/2} \|f\|_{B([0, T]; C^\theta(\overline \Omega))} \\ \\
  + \int_{\R^+ \times \R^+} e^{-t \tau} \frac{t^{\frac{\alpha \theta}{2}} \tau^\alpha}{\tau^{2\alpha} + \xi^2} dt d\tau \|\gamma f\|_{C^{\frac{\alpha \theta}{2}}([0, T]; C^\theta(\overline \Omega))}) \\ \\
  = C_8 \xi^{-\theta/2}(\|f\|_{B([0, T]; C^\theta(\overline \Omega))} + \|\gamma f\|_{C^{\frac{\alpha \theta}{2}}([0, T]; C^\theta(\overline \Omega))}). 
\end{array}
$$ 
So (\ref{eq3.6}) holds and the assertion is completely proved. 
\end{proof}

Now we consider the case $g \equiv 0$. We begin with the following

\begin{lemma}\label{le3.9}
Suppose that (A1)-(A4) hold. Moreover, 

(I) $f \in C([0, T]; C(\overline \Omega)) \cap B([0, T]; C^\theta(\overline \Omega))$; 
 
 (II) $u_0 \in C^{2+\theta}_0(\overline \Omega)$; 
 
 (III) $\gamma f  \in C^{\frac{\alpha \theta}{2}} ([0, T]; C(\partial \Omega))$; 
 
 (IV) $\gamma[Au_0 + f(0))= 0$. 
 
 \medskip
 If $t \in (0, T]$, we set
 $$
 u(t) := u_0 + \int_0^t T(t-s) [f(s) + Au_0] ds,
 $$
 with $T(t)$ as in (\ref{eq3.4}).  Then $u$ satisfies (B1)-(B2) and is a solution to (\ref{eq1.1}), with $g \equiv 0$ and, in case $\alpha \in (1, 2)$, $u_1 = 0$. 
\end{lemma}

\begin{proof} We set, for $t \in (0, T]$, 
$$
v(t):= \int_0^t T(t-s) [f(s) + Au_0] ds. 
$$
Then, by Proposition \ref{pr3.1}, $v$ is a strict solution to 
$$
 B^\alpha_{C(\overline \Omega)} v(t) = Av(t) + f(t) + Au_0, \quad t \in [0, T], 
 $$
 and, moreover, $B^\alpha_{C(\overline \Omega)} v \in B([0, T]; C^\theta (\overline \Omega))$ and $v \in B([0, T]; C^{2+\theta}(\overline \Omega))$. We deduce that $u(0) = u_0$, in case $\alpha \in (0, 1)$, 
 $D_t u(0) = D_tv(0) = 0$ (by Lemma \ref{le1.3}), $\mbD^\alpha_{C(\overline \Omega)} u$ is defined and, for $t \in [0, T]$, 
 $$
 \mbD^\alpha_{C(\overline \Omega)} u(t) = B_{C(\overline \Omega)}^\alpha v(t) = Av(t) + Au_0 + f(t) = Au(t) + f(t). 
 $$
 The fact that $u$ satisfies (B1)-(B2) is clear. 
\end{proof}

\begin{lemma} \label{le3.10}
Suppose that (A1)-(A4) hold, with $\alpha \in (1, 2)$. Let $u_1 \in C^{\theta + 2(1 - \frac{1}{\alpha})}_0(\overline \Omega)$. We adopt again the convention (\ref{eq3.7}) and set, for $t$ in $(0, T]$
$$
u(t):= \frac{1}{2\pi i} \int_{\Gamma} e^{\lambda t} \lambda^{\alpha -2} (\lambda^\alpha - A)^{-1} u_1 d\lambda.
$$
Then $u$ satisfies (B1)-(B2) and is a solution to (\ref{eq1.1}), with $u_0 = 0$, $f \equiv 0$, $g \equiv 0$. 
\end{lemma} 

\begin{proof} We put, for $|\lambda| \geq R^{1/\alpha}$, $|Arg(\lambda)| \leq \frac{\pi - \omega}{\alpha}$, 
$$
F(\lambda):= \lambda^{\alpha -2} (\lambda^\alpha - A)^{-1} u_1.  
$$
As $u_1$ belongs to the closure of $D(A)$ in $C(\overline \Omega)$ (because $\gamma u_1 = 0$), we have
$$
\lim_{|\lambda| \to \infty} \lambda^2 F(\lambda) = \lim_{|\lambda| \to \infty} \lambda^{\alpha} (\lambda^\alpha - A)^{-1} u_1 = u_1. 
$$
We deduce from Proposition \ref{pr7A} that $u$ belongs to $D(B_{C(\overline \Omega)})$ and $B_{C(\overline \Omega)}u(0) = u_1$,  so that $u(0) = 0$, $D_tu(0) = u_1$. If $t \in (0, T]$, we have
$$
u(t) - tu_1 = \frac{1}{2\pi i} \int_{\Gamma} e^{\lambda t} [\lambda^{\alpha -2} (\lambda^\alpha - A)^{-1} u_1 - \lambda^{-2} u_1] d\lambda = \frac{1}{2\pi i} \int_{\Gamma} e^{\lambda t} \lambda^{ -2} A(\lambda^\alpha - A)^{-1} u_1 d\lambda. 
$$
By Proposition \ref{pr2.4} (III), we have 
$$
C^{\theta + 2(1 - \frac{1}{\alpha})}_0(\overline \Omega) = (C(\overline \Omega), D(A))_{\frac{\theta}{2} + 1 - \frac{1}{\alpha},\infty}, 
$$
so that, by Proposition \ref{pr2.1A} (III), we have
$$
\|A(\lambda^\alpha - A)^{-1} u_1\|_{C(\overline \Omega)} \leq C |\lambda|^{1 - \alpha (\frac{\theta}{2} + 1)}. 
$$
We deduce that 
$$
\lim_{|\lambda| \to \infty} |\lambda|^{1+\alpha} \|\lambda^{ -2} A(\lambda^\alpha - A)^{-1} u_1\|_{C(\overline \Omega)} = 0. 
$$
So by Proposition \ref{pr7A} $\mbD^\alpha_{C(\overline \Omega)}u$ is defined. Moreover,  $\forall t \in [0, T]$
$$
\mbD^\alpha_{C(\overline \Omega)}u(t) = \frac{1}{2\pi i} \int_{\Gamma} e^{\lambda t} \lambda^{\alpha -2} A(\lambda^\alpha - A)^{-1} u_1 d\lambda = Au(t). 
$$
If remains to show that $Au$ is bounded with values in $C^\theta(\overline \Omega)$. To this aim, we introduce the operator $A_\theta$, defined as follows: 
$$
\left\{\begin{array}{l}
D(A_\theta): = \{u \in C^{2+\theta}_0(\overline \Omega) : Au \in C^{\theta}_0(\overline \Omega)\}, \\ \\
A_\theta u = Au, \quad u \in D(A_\theta). 
\end{array}
\right. 
$$
As 
$$C_0^\theta (\overline \Omega) = (C(\overline \Omega), D(A))_{\theta/2,\infty} = (C(\overline \Omega), D(A^2))_{\theta/4,\infty}, $$
$A_\theta$, as unbounded operator in $C_0^\theta (\overline \Omega)$, can be taken as operator $B$ in Proposition \ref{pr2.1A}. We have that  
$$D(A_\theta) = (D(A), D(A^2))_{\theta/2,\infty} = (C(\overline \Omega), D(A^2))_{\frac{2+\theta}{4}, \infty},$$ 
with equivalent norms. So, by Proposition \ref{pr2.1A} and the reiteration theorem, we deduce
$$
(C_0^\theta(\overline \Omega), D(A_\theta))_{1 - \frac{1}{\alpha},\infty} = (C(\overline \Omega), D(A^2))_{\frac{\theta}{4} + \frac{1}{2}(1-\frac{1}{\alpha}),\infty} = (C(\overline \Omega), D(A))_{\frac{\theta}{2} + 1-\frac{1}{\alpha},\infty} = C_0^{\theta + 2(1-\frac{1}{\alpha})}(\overline \Omega).
$$
We deduce that, if $\lambda \in \Gamma$, 
$$
\|\lambda^{\alpha -2} A(\lambda^\alpha - A)^{-1} u_1\|_{C^\theta(\overline \Omega)} = \|\lambda^{\alpha -2} A(\lambda^\alpha - A)^{-1} u_1\|_{C_0^{\theta + 2(1-\frac{1}{\alpha})}  (\overline \Omega)} \leq 
C |\lambda|^{-1} \|u_1\|_{C_0^{\theta + 2(1-\frac{1}{\alpha})}(\overline \Omega)}, 
$$
so that, by Remark \ref{re2.2A}, $Au \in B([0, T]; C^\theta_0(\overline \Omega))$ and $u \in B([0, T]; C^{2+\theta}(\overline \Omega))$. 
\end{proof}

\begin{corollary}\label{co3.1}
Suppose that (A1)-(A4) are fulfilled. Consider system (\ref{eq1.1}) in case $g \equiv 0$. Then the following conditions are necessary and sufficient, in order that there exists a unique solution $u$
satisfying (B1)-(B2): 
 
 (I) $f \in C([0, T]; C(\overline \Omega)) \cap B([0, T]; C^\theta(\overline \Omega))$.
 
 (II) $u_0 \in C_0^{2+\theta}(\overline \Omega)$ and, in case $\alpha \in (1, 2)$, $u_1 \in C_0^{\theta + 2(1 - \frac{1}{\alpha})}(\overline \Omega)$. 
 
 (III) $\gamma f \in C^{\frac{\alpha \theta}{2}} ([0, T]; C(\partial \Omega))$. 
 
 (IV) $\gamma[A(\cdot,D_x) u_0 + f(0)] = 0$. 

\end{corollary} 

\begin{proof} The necessity of conditions (I)-(IV) follows from Lemma \ref{le3.2}. The uniqueness of a solution follows from Proposition \ref{pr2.5}. Concerning the existence, it suffices to take the sum of the solution of  (\ref{eq1.1}) with $g \equiv 0$, $u_1 = 0$ (in case $\alpha > 1$) with the solution of  (\ref{eq1.1})  with $u_0 = 0$, $f \equiv 0$, $g \equiv 0$, the existence of which follows from Lemma \ref{le3.9} and \ref{le3.10}. 

\end{proof}

{\bf Proof of Theorem \ref{th1.1}.} The uniqueness of a solution follows from Proposition \ref{pr2.5}. 

We prove the existence. Let $R$ be the operator introduced in Lemma \ref{le2.1} (IV). We set
\begin{equation}\label{eq3.8A}
v(t):= Rg(t), \quad t \in [0, T]. 
\end{equation}
Then $v \in C([0, T]; C^2(\overline \Omega)) \cap B([0, T]; C^{2+\theta}(\overline \Omega))$. We set $h:= \mathbb{D}^\alpha_{C(\partial \Omega)}g$. Then, by (IV), if we put $u_1 = 0$ in case $\alpha \in (0, 1]$, we have
$$
g(t) = \gamma u_0 + t \gamma u_1 + \frac{1}{\Gamma (\alpha)} \int_0^t (t-s)^{\alpha -1} h(s) ds, \quad t \in [0, T]. 
$$
We deduce 
$$
v(t) = R\gamma u_0 + t R\gamma u_1 + \frac{1}{\Gamma (\alpha)} \int_0^t (t-s)^{\alpha -1} Rh(s) ds, \quad t \in [0, T]. 
$$
so that $\mathbb{D}^\alpha_{C(\overline \Omega)}v$ exists and coincides with $R\mathbb{D}^\alpha_{C(\partial \Omega)}g$, implying that $\mathbb{D}^\alpha_{C(\overline \Omega)}v$ belongs to $C([0, T]; C(\overline \Omega))
 \cap B([0, T]; C^\theta(\overline \Omega))$. From this we deduce (applying Lemma \ref{le1.4} (III)) that $v \in C^{\frac{\alpha \theta}{2}}([0, T]; C^2(\overline \Omega))$. 
 
 Now we take, as new unknown, $w:= u - v$. $w$ should solve the system
 
  \begin{equation} \label{eq3.8}
\left\{
\begin{array}{l}
\mathbb{D}^\alpha_{C(\overline \Omega)} w(t,x) =  A(x,D_x) w(t,x) + f(t,x) - \mathbb{D}^\alpha_{C(\overline \Omega)} v(t,x) + A(x,D_x) v(t,x) , \quad t \in [0, T], x \in \Omega, \\ \\
w(t,x') = 0, \quad (t,x') \in [0, T] \times \partial \Omega, \\ \\
D_t^k w(0,x) = u_k(x) - R(\gamma u_k)(x), \quad x \in \overline \Omega, k \in \N_0, k < \alpha. 
\end{array}
\right. 
\end{equation}
It is easily seen that Corollary \ref{co3.1} is applicable to system (\ref{eq3.8}). So there exist a solution $w$ satisfying (B1)-(B2). If we put $u:= v + w$, we obtain a solution of (\ref{eq1.1}), satisfying (B1)-(B2).

\section{Proof of Theorem \ref{th1.2}} \label{se4}

\setcounter{equation}{0}

We begin by showing that conditions (I)-(V) in Theorem \ref{th1.2} are necessary to get the conclusion. If $u$ has the required regularity, it satisfies also (B1)-(B2). So conditions (I)-(VI) in the statement of Theorem \ref{th1.1} are all necessary. It is clear that, necessarily, $f$ should belong to $C^{\frac{\alpha \theta}{2}, \theta}([0, T] \times \overline \Omega)$. Moreover, as 
$$
\mathbb{D}^\alpha_{C(\partial \Omega)}g = \gamma \mathbb{D}^\alpha_{C(\overline \Omega)}u, 
$$
necessarily, $\mathbb{D}^\alpha_{C(\partial \Omega)}g$ belongs to $C^{\frac{\alpha \theta}{2}, \theta}([0, T] \times \partial \Omega)$. 

Now we show that these conditions are also sufficient. Following the argument in the proof of Theorem \ref{th1.1}, we define $v$ as in (\ref{eq3.8A}). Then 
$$v \in C([0, T]; C^2(\overline \Omega)) \cap B([0, T]; C^{2+\theta}(\overline \Omega)), $$
$\mathbb{D}^\alpha_{C(\overline \Omega)}v$ exists and coincides with $R\mathbb{D}^\alpha_{C(\partial \Omega)}g$, implying that $\mathbb{D}^\alpha_{C(\overline \Omega)}v$ belongs to 
$C^{\frac{\alpha \theta}{2}, \theta}([0, T] \times \overline \Omega)$. By Lemma \ref{le1.4}(III), $v \in C^{\frac{\alpha \theta}{2}}([0, T]; C^2(\overline \Omega))$, so that $A(\cdot, D_x) v$ belongs to 
$C^{\frac{\alpha \theta}{2}, \theta}([0, T] \times \overline \Omega)$. Subtracting $v$ to $u$, we are reduced to consider system (\ref{eq3.8}). Arguing as in the proof of Theorem \ref{th1.1}, we see that its solution $w$
 satisfies (B1)-(B2). Moreover, $f - \mathbb{D}^\alpha_{C(\overline \Omega)} v + A(\cdot,D_x) v$ belongs to $C^{\frac{\alpha \theta}{2}}([0, T]; C(\overline \Omega))$, $u_0 - R(\gamma u_0) \in D(A)$, if $\alpha \in (1, 2)$, 
 $u_1 - R(\gamma u_1) \in C_0^{\theta + 2(1 - \frac{1}{\alpha})}(\overline \Omega)$, 
 $$A(u_0  - R\gamma u_0) + f(0) - \mathbb{D}^\alpha_{C(\overline \Omega)}v(0) + A(\cdot,D_x) v(0) = A(\cdot,D_x) u_0 + f(0) -  \mathbb{D}^\alpha_{C(\overline \Omega)}v(0) \in C^\theta (\overline \Omega),$$
 $$\gamma [A(\cdot,D_x) u_0 + f(0) -  \mathbb{D}^\alpha_{C(\overline \Omega)}v(0)] = \gamma [A(\cdot,D_x) u_0 + f(0)] - \mathbb{D}^\alpha_{C(\partial \Omega)}g(0) = 0. 
 $$
 We deduce from Corollary \ref{co2.3}  that $\mathbb{D}^\alpha_{C(\overline \Omega)}w$ and $Aw = A(\cdot, D_x) w$ belong to $C^{\frac{\alpha \theta}{2}}([0, T]; C(\overline \Omega))$, so that $w$ satisfies (D1)-(D2). The conclusion is that $u = v + w$ satisfies (D1)-(D2). 
 
 $\square$
 
 \begin{remark}
 {\rm In case $\alpha = 1$, (D1)-(D2) imply that $u$ belongs to $C^{1+\frac{\theta}{2}}([0, T]; C(\overline \Omega))$, so that $u$ belongs to $C^{1+\frac{\theta}{2}, 2+\theta}([0, T] \times \overline \Omega )$. This suggest that in the general case $u$ should belong to $C^{\alpha+\frac{\alpha\theta}{2}, 2+\theta}([0, T] \times \overline \Omega )$.
 
  In case $\alpha \neq 1$, 
 $u$ may satisfy (D1)-(D2) without belonging to any space $C^{\alpha + \epsilon}([0, T]; C(\overline \Omega))$ for any $\epsilon$ positive. Consider the following example: let $\alpha \in (0, 2) \setminus \{1\}$.  
 Fix $f_0$ in $C^{2+\theta}_0(\overline \Omega) \setminus
 \{0\}$, $\theta \in (0, 2) \setminus \{1\}$,  and define
 $$
 \left\{ \begin{array}{l}
 u : [0, T] \times \overline \Omega \to \C, \\ \\
 u(t,x) = \frac{t^\alpha}{\Gamma (\alpha +1)} f_0(x). 
 \end{array}
  \right. 
 $$
 Then $u$ solves (\ref{eq1.1}), if we take $f(t,x) = f_0(x) - \frac{t^\alpha}{\Gamma (\alpha +1)} [A(\cdot,D_x) f_0](x)$, $g \equiv 0$, $D_t^k u(0,\cdot) = 0$ if $k \in \N_0$, $k < \alpha$. It is easily seen that in this case the assumptions (I)-(V) of Theorem \ref{th1.2} are satisfied. However, $u$ does not belong to any space $C^{\alpha + \epsilon}([0, T]; C(\overline \Omega))$, for any $\epsilon$ positive. 
 
 Nevertheless, let $v \in D(B_X^\alpha)$ be such that $B_X^\alpha v \in C^{\beta}([0, T]; X)$, with $\alpha + \beta, \beta \in \R^+ \setminus \N$. Then $v$ can be represented in the form
 $$
 v(t) = \sum_{k \in \N_0, k < [\beta]} t^{k+\alpha} v_k + w(t),
 $$
 with $v_k \in X$ for each $k$, $w \in C^{\alpha + \beta}([0, T]; X)$, $w^{(j)}(0) = 0$, for each $j$ in $\N_0$, $j < \alpha + \beta$ (see \cite{Gu3}, Proposition 12). 
 We deduce that in the situation of Theorem \ref{th1.2}, at least in case $\alpha(1 + \frac{\theta}{2}) \not \in \N_0$, the solution $u$ can be written in the form
 $$
 u(t) = U(t) + t^\alpha v_0, 
 $$
 with $v_0 \in C(\overline \Omega)$, $U \in C^{\alpha+\frac{\alpha\theta}{2}}([0, T]; C(\overline \Omega))$. 
 
 }
 \end{remark}

\end{document}